\documentclass[reqno]{amsart}
\usepackage{amssymb,latexsym}
\usepackage{amsmath,amssymb,amsfonts,amsthm}
\usepackage[latin1]{inputenc}
\usepackage[T1]{fontenc}
\usepackage{bbm}

\allowdisplaybreaks[3]

\theoremstyle{plain}
\newtheorem{prop}{Proposition}[section]
\newtheorem{proposition}[prop]{Proposition}
\newtheorem{theorem}[prop]{Theorem}
\newtheorem{cor}[prop]{Corollary}%[section]
\newtheorem{lemma}[prop]{Lemma}%[section]
\theoremstyle{definition}
\newtheorem{definition}[prop]{Definition}
\newtheorem*{defi*}{Definition}
\newtheorem{example}[prop]{Example}

\newtheorem{remark}[prop]{Remark}
\newtheorem{assumption}[prop]{Assumption}%[section]

\theoremstyle{remark}

\newtheorem{claim}{Claim}

\newcommand{\Rd}{\mathbb{R}^{d}}
\newcommand{\hs}{[0,\infty) \times \mathbb{R}^{d}}

\newcommand{\BN}{\mathbb{N}}

\newcommand{\BR}{\mathbb{R}}

\newcommand{\CB}{\mathcal{B}}

\newcommand{\CF}{\mathcal{F}}

\newcommand{\CM}{\mathcal{M}}

\newcommand{\CK}{\mathcal{K}}

\newcounter{TempForOutput}

\begin{document}
\title{Brownian Motion with Singular Time-Dependent Drift}
\author{Peng Jin}
\address{Fakult\"at f\"ur Mathematik und Naturwissenschaften \\
                 Bergische Universit\"at Wuppertal \\ Gau\ss stra\ss e 20, 42119 Wuppertal, Germany}
\thanks{Research supported by the DFG through IRTG-1132}
\begin{abstract}In this paper we study weak solutions for the following type of stochastic differential equation
\[
\begin{cases}dX_{t}=dW_{t}+b(t, X_{t})dt, & \ t\ge s,\\
   \   X_{s}=x,
    \end{cases}
\]
where $b: [0,\infty) \times \mathbb{R}^{d} \to \mathbb{R}^{d}$ is a measurable drift, $W=(W_{t})_{t \ge 0}$ is a $d$-dimensional Brownian motion and $(s,x)\in [0,\infty) \times \mathbb{R}^{d}$ is the  starting point. A solution $X=(X_t)_{t \ge s}$ for the above SDE is called a Brownian motion with time-dependent drift $b$ starting from $(s,x)$. Under the assumption that $|b|$ belongs to the forward-Kato class $\mathcal{F} \mathcal{K}_{d-1}^{\alpha}$ for some $\alpha \in (0,1/2)$, we prove that the above SDE has a unique weak solution for every starting point $(s,x)\in [0,\infty) \times \mathbb{R}^{d}$.
\end{abstract}

\maketitle

%%%%%% Section 1: Introduction
\section{Introduction}
In this paper we consider weak solutions to the following stochastic differential equation
\begin{equation}\label{SDE1}
\begin{cases}dX_{t}=dW_{t}+b(t, X_{t})dt, & \ t\ge s,\\
   \   X_{s}=x,
    \end{cases}   % \cases{
  %dX_{t}=dW_{t}+b(t, X_{t})dt, \quad &  $t\ge s$, \cr
 %\hspace*{1.7mm}  X_{t}=x, \quad & $0 \le t \le s$, \cr
%}
\end{equation}
where $b: \hs \to \Rd$ is measurable, $W=(W_{t})_{t \ge 0}$ is a $d$-dimensional Brownian motion and $(s,x)\in \hs$ is the  starting point. Throughout this paper we assume that $d \ge 3$. A solution $X=(X_t)_{t \ge s}$ for the SDE (\ref{SDE1}) is called a Brownian motion with time-dependent drift $b$ starting from $(s,x)$. Since the drift $b$ is not necessarily locally bounded, we emphasize that solutions of (\ref{SDE1}) are supposed to fulfill the integrability condition
\[
\int_s^t|b(u,X_u)|du<\infty \quad \mbox{a.s.}, \quad \forall t \ge s.
\]

In order to get weak solutions to (\ref{SDE1}), the most straightforward approach is to use the Girsanov transformation. This approach has been investigated in many papers (see for example, \cite{MR774179,MR1104660,MR1246978}). In \cite{MR1246978}, Stummer gave several examples of singular drift $b$ such that the Girsanov transformation is applicable and thus weak solutions to (\ref{SDE1}) exist. It should be noted that if a weak solution to (\ref{SDE1}) is obtained through Girsanov transformation, then its law on the path space would be absolutely continuous with respect to the law of the Brownian motion.

Besides using Girsanov transformation, there are some analytical approaches to solve (\ref{SDE1}). In the case when the drift $b(t,x)=b(x)$ is time-independent, Bass and Chen \cite{MR1964949} considered the following SDE
\begin{equation}\label{eqsde2}
\begin{cases}dX_{t}=dW_{t}+b( X_{t})dt, \\
   \   X_{0}=x,
    \end{cases}%\cases{
 % dX_{t}=dW_{t}+b(X_{t})dt, \quad &  $t\ge 0$, \cr
 %\hspace*{1.7mm}  X_{0}=x. \quad &  \cr
%}
\end{equation}
and proved that if $|b|$ belongs to the Kato class $\CK_{d-1}$, namely
\begin{equation}\label{katodm1}
    \lim_{r \to 0} \sup_{x \in \Rd}\int_{B(x;r)}\frac{|b(y)|}{|x-y|^{d-1}}dy=0,
\end{equation}
then (\ref{eqsde2}) has a unique weak solution. Their method is based on the construction of  the resolvent $G^{\lambda}$ ($\lambda>0$) of the desired process $(X_{t})_{t \ge 0}$. Let $R^{\lambda}$ denote the resolvent operator of Brownian motion. When $|b| \in \CK_{d-1}$, the generator $L=\triangle/2 + b \cdot \nabla$ of $(X_{t})_{t \ge 0}$ can be considered as a small perturbation of $\triangle/2$ and intuitively
\begin{equation}\label{idforres}
 G^{\lambda}=\Big (\lambda- \frac{1}{2}\triangle - b\cdot \nabla \Big )^{-1}=R^{\lambda}\Big ( \sum^{\infty}_{j=0}(BR^{\lambda})^{j}\Big),
 \end{equation}
where $B$ denotes the operator $b \cdot \nabla$. With the help of some gradient estimates on $R^{\lambda}$, the identity (\ref{idforres}) was rigorously established in \cite{MR1964949}.  We should point out that the above mentioned result of \cite{MR1964949} hold in a more general setting and is actually valid for the case when the drift $b$ is a Radon measure that satisfies (\ref{katodm1}), although in this case the notion of a solution to (\ref{eqsde2}) has to be defined in a more general sense. Later, analytical and probabilistic properties of the solution $(X_{t})_{t \ge 0}$ to (\ref{eqsde2}) with a drift $|b| \in \CK_{d-1}$ were investigated by Kim and Song, see \cite{MR2247841,MR2317765,MR2319644,MR2411521}; among many other things, they obtained two-sided Gaussian estimates for the heat kernel of $(X_{t})_{t \ge 0}$.

In this paper we shall apply the method of \cite{MR1964949} to treat the SDE (\ref{SDE1}), where the drift is singular and time-dependent. From Bass and Chen's work, we know that a possible way to get a weak solution to (\ref{SDE1}) is to construct the resolvent operator of the solution. Since the drift $b(t,\cdot)$ in (\ref{SDE1}) is time-dependent, we have to consider the space-time resolvent $S^{\lambda}$, namely
\begin{equation}
 S^{\lambda}=\Big (\lambda-\big( \frac{\partial }{\partial t}+\frac{1}{2}\triangle +b(t,\cdot) \cdot \nabla \big)\Big )^{-1}.
 \end{equation}
The generator  of the space-time process $(X_{t},t)_{t \ge s}$ is given by $\partial/\partial t+\triangle/2 +b(t,\cdot) \cdot \nabla$ and could be considered as a perturbation of $\partial/\partial t+\triangle/2 $, which is the generator of the process $(W_{t},t)_{t \ge s}$. Therefore, we expect to obtain a similar expression of $S^{\lambda}$ like the identity (\ref{idforres}). However, we first have to identify the class of drifts for which the term $b(t,\cdot) \cdot \nabla$ is ``small enough" compared to $\partial/\partial t+\triangle/2$ and the perturbation argument works.

We now state our assumption on the drift $b$ in (\ref{SDE1}). We assume $|b|$ to be in the forward-Kato class $\CF \CK_{d-1}^{\alpha}$ for some $\alpha \in (0,1/2)$, that is
\begin{equation}\label{assump}
    \lim_{h \to 0}N^{\alpha,+}_{h}(|b|)=0,
\end{equation}
where
\[
    N^{\alpha,+}_{h}(|b|):= \sup_{(s,x)\in \hs}\int_{s}^{s+h}\int_{\mathbb{R}^{d}}\frac{1}{(t-s)^{\frac{d+1}{2}}}\exp(-\alpha \frac{|y-x|^{2}}{t-s})|b(t,y)|dydt.
\]
It's worth noting that the forward-Kato class $\CF \CK_{d-1}^{\alpha}$ includes the (time-independent) Kato class $\CK_{d-1}$ and the parabolic Kato class defined in \cite[Definition 1.1]{MR1457736} (see also \cite[Definition 3.1]{MR2853532}). Under the above assumption on the drift $b$, we will prove in this paper that the SDE (\ref{SDE1}) has a unique weak solution for every starting point $(s,x)$. The point is that if the condition (\ref{assump}) is satisfied, then the perturbation method mentioned above to construct the resolvent $S^{\lambda}$ of the process $(t, X_{t})_{t \ge s}$  applies. We should remark that our assumption (\ref{assump}) allows us to do a perturbation on the space-time resolvent of Brownian motion, but is generally not strong enough to enable us to carry out a perturbation on the Brownian heat kernel. To do a drift perturbation on the Brownian heat kernel $p(s,x;t,y)$, one has to deal with two singularities of $p(s,x;t,y)$ both at $s$ and $t$, and thus needs a stronger assumption like the one used in \cite{MR2853532,MR1457736}.

In this paper we only consider weak solutions to (\ref{SDE1}). However, the existence of strong solutions to SDEs with a singular drift term is also an interesting problem and has been well-studied.  Concerning the SDE (\ref{SDE1}), Krylov and R\"ockner \cite{MR2117951} proved that if $b$ is locally in $L^{p,q}$ (see Example \ref{examplelpq} below for a definition) with $p \ge 2$ and $d/2p+1/q<1/2$, then (\ref{SDE1}) has a unique strong solution up to an explosion time. For the case of non-constant Sobolev diffusion coefficients, see \cite{MR2820071}.

A similar problem related to this paper is to consider an $\alpha$-stable process with singular drift. Recently, Chen and Wang \cite{ChenZ.Q.arXiv1309.6414} studied rotationally symmetric $\alpha$-stable process with a drift belonging to the Kato class $\CK_{d,\alpha-1}$ (see \cite[Definition 1.1]{ChenZ.Q.arXiv1309.6414}); they proved the existence and uniqueness, in the weak sense, of such a process and established sharp two-sided estimates for its heat kernel. Sharp two-sided Dirichlet heat kernel estimates for such a drifted $\alpha$-stable process were derived in \cite{MR3050510}. As shown in \cite{MR3192504}, similar results hold when the drift is not a vector-valued function but a signed measure belonging to $\CK_{d,\alpha-1}$. We refer also to \cite{MR2945756,MR3127913} where the case of a more general $\alpha$-stable process was discussed.

The rest of the paper is organized as follows. In Section 2 we recall the definitions of the martingale and local martingale problems and their connections to weak solutions of SDEs. In Section 3 we collect some properties of the forward-Kato class $\CF \CK_{d-1}^{c}$. In Section 4 we prove some gradient estimates for the space-time resolvent $R^{\lambda}$ of Brownian motion. In Section 5 we prove the local existence and uniqueness of weak solutions to (\ref{SDE1}). In Section 6 we obtain the global existence and uniqueness of weak solutions to (\ref{SDE1}). Finally, we fix some notation used below. For a bounded function $g$ on $\hs$ we write $\|g\|_{\infty}:=\sup_{(s,x) \in \hs} |g(s,x)|$.

%%%%%% Section 2: Preliminaries
\section{Preliminaries}
As well-known, weak solutions to SDEs are equivalent to solutions to the corresponding local martingale problem of Stroock and Varadhan. As compared with dealing with weak solutions to SDEs directly, using the local-martingale-problem approach has several advantages. One advantage is the availability of the localization technique, which is essential for this paper.

Let
\[
    L_{t}=\frac{1}{2}\triangle + b(t,\cdot) \cdot \nabla,
\]
where $\triangle$ and $\nabla$ are the Laplacian and the gradient operator, respectively, with respect to the spatial variable $x$. Let
 $\Omega=C([0,\infty); \mathbb{R}^{d})$ be the space of continuous trajectories from $[0,\infty)$ into $\Rd$. Given $t\ge 0$ and $\omega \in \Omega$, define $X_{t}(\omega):=\omega_t$. Let
\[
    \CM_{t}:=\sigma (X_{s}: 0 \le s \le t),\]
and
\[
    \CM:= \sigma \big( \cup_{t \ge 0} \CM_{t} \big ).\]

Given $(s,x)\in [0,\infty) \times \mathbb{R}^{d}$, a solution to the local martingale problem for $L_{t}$ starting from $(s,x)$ is a probability measure $\mathbf{P}$ on $(\Omega=C([0,\infty);\mathbb{R}^{d}),\mathcal{M})$ with the following properties:
\begin{equation}\label{wsintegrcondi}
    \mathbf{P}(X_{t}=x,\ \forall t \le s)=1, \quad \mathbf{P}\Big(\int_s^t|b(u,X_u)|du<\infty, \ \forall t \ge s\Big)=1 \end{equation}
 and
\begin{equation}\label{eq2defimp}
    f(X_{t})-\int^{t}_{s}L_{u}f(X_{u})du
\end{equation}
is a $(\mathbf{P}, \mathcal{M}_{t})$ local martingale after time $s$ for all $f \in C^{\infty}_{0}(\mathbb{R}^{d})$.

If a probability measure $\mathbf{P}$ on $(\Omega=C([0,\infty);\mathbb{R}^{d}),\mathcal{M})$ satisfies (\ref{wsintegrcondi}) and is such that the process defined by (\ref{eq2defimp}) is a $(\mathbf{P}, \mathcal{M}_{t})$ martingale after time $s$ for all $f \in C^{\infty}_{0}(\mathbb{R}^{d})$, then $\mathbf{P}$ is called a solution to the martingale problem for $L_{t}$ starting from $(s,x)$.

In our case, since the second order term in $L_{t}$ is $\triangle/2$, it follows from \cite[Proposition 4.11]{MR1121940} that the martingale problem and the local martingale problem for $L_{t}$ are equivalent.

We say that the martingale problem for $L_{t}$ is well-posed if, for each $(s,x) \in \hs$, there is exactly one solution to that martingale problem starting from $(s,x)$.

%%%%%% Section 3: Forward-Kato class
\section{Forward-Kato Class $\CF \CK_{d-1}^{c}$}
In this section we give some examples of functions which belong to the forward-Kato class $\CF \CK_{d-1}^{c}$. Some properties of $\CF \CK_{d-1}^{c}$ are also discussed.
\begin{definition}\label{defifk}Let $c>0$ be a constant. A measurable function $f$ on $\hs$ is said to be in  the forward-Kato class $\CF \CK_{d-1}^{c}$ if
\[
    \lim_{h \to 0}N^{c,+}_{h}(f)=0,\]
    where
\begin{equation}\label{normnch}
    N^{c,+}_{h}(f):= \sup_{(s,x)\in \hs}\int_{s}^{s+h}\int_{\mathbb{R}^{d}}\frac{1}{(t-s)^{\frac{d+1}{2}}}\exp(-c \frac{|y-x|^{2}}{t-s})|f(t,y)|dydt.
\end{equation}
\end{definition}

The forward-Kato class $\CF \CK_{d-1}^{c}$ includes several important classes of functions. If a measurable function $f$ on $\hs$ is bounded, then it is obvious that $f \in \CF \CK^{c}_{d-1}$ for any $c>0$. The following are some other examples of functions which are in $\CF \CK_{d-1}^{c}$.

\begin{example}\label{exam1.1.1}

Let a measurable function $f:\hs \to \BR$ be time-independent, namely $f(t,x)=f(x)$, and $|f(x)| \in \CK_{d-1}$ (see the condition (\ref{katodm1}) in the introduction). Then $f \in  \CF \CK^{c}_{d-1} $ for any $c >0$. The reader is referred to \cite[Proposition 2.3]{MR2247841} for a proof of this fact.
\end{example}

\begin{example}\label{examplelpq}
For $p, q \in [1, \infty]$ we set $L^{p}=L^{p}(\Rd)$, $L^{p,q}=L^{q}(\mathbb{R}, L^{p})$. If a measurable function $f$ on $\hs$ has compact support and $f \in L^{p,q}$ (here $f$ is set to be 0 on $\mathbb{R}^{d+1} \setminus \hs$) with $d/2p+1/q<1/2$, then $f \in \CF \CK^{c}_{d-1}$ for any $c>0$, see \cite[Proposition 2.1]{MR1457736} for a proof.
\end{example}

\begin{example}\label{exam1.1.3}
Let $c>0$ be a constant. For $(t,y) \in \hs$ define
\begin{displaymath}
f(t,y):= \left\{ \begin{array}{ll}
-\frac{1}{(1-t)^{1/2}\ln(1-t)}, \ & \textrm{if $\ \frac{1}{2}\le t<1$ and $|y|\le 3d(1-t)^{\frac{1}{2}}$}, \\
0, & \textrm{otherwise}.
\end{array} \right .
\end{displaymath}
Consider $\tilde{f}$, the ``time-reversal''  of $f$, namely
\begin{displaymath}
\tilde{f}(t,y):= \left\{ \begin{array}{ll}
f(1-t,y), \ & \textrm{if $\ 0\le t\le1$ and $y\in \Rd$}, \\
0, & \textrm{otherwise}.
\end{array} \right .
\end{displaymath}
Then
\begin{equation}\label{appendixclaim1}
f \in \CF \CK^{c}_{d-1} \quad \mbox{but} \quad \tilde{f} \notin \CF \CK^{c}_{d-1}.
\end{equation}
The proof of (\ref{appendixclaim1}) is technical and is put in the appendix. This example shows that our forward-Kato class $\CF \CK_{d-1}^{c}$ is strictly larger than the parabolic Kato classes defined in \cite[Definition 1.1]{MR1457736} and \cite[Definition 3.1]{MR2853532}.
\end{example}

We next give some properties of the class $\CF \CK_{d-1}^{c}$. These properties will be used very often in subsequent sections.

\begin{proposition}\label{prop1.1.1}Suppose that $f \in \CF \CK^{c}_{d-1}$. Then for any $h>0$ we have $N_{h}^{c,+}(f)< \infty$, where $N_{h}^{c,+}(f)$ is defined in (\ref{normnch}) .
\end{proposition}
\begin{proof}Let $x,y \in \Rd$ and $s<s_{1} <t$. Then we have the following inequality
\begin{align}
&\int_{\mathbb{R}^{d}}\frac{(2c)^{\frac{d}{2}}}{(2\pi)^{\frac{d}{2}}(s_{1}-s)^{\frac{d}{2}}}\exp(-c  \frac{|z-x|^{2}}{s_{1}-s})\frac{(2c)^{\frac{d}{2}}}{(2\pi)^{\frac{d}{2}}(t-s_{1})^{\frac{d+1}{2}}}\exp(-c \frac{|y-z|^{2}}{t-s_{1}})dz \notag\\
\ge & \frac{1}{(t-s)^{\frac{1}{2}}}\int_{\mathbb{R}^{d}}\frac{(2c)^{\frac{d}{2}}}{(2\pi)^{\frac{d}{2}}(s_{1}-s)^{\frac{d}{2}}}\exp(-c  \frac{|z-x|^{2}}{s_{1}-s})\frac{(2c)^{\frac{d}{2}}}{(2\pi)^{\frac{d}{2}}(t-s_{1})^{\frac{d}{2}}}\exp(-c \frac{|y-z|^{2}}{t-s_{1}})dz \notag \\
\ge & \frac{1}{(t-s)^{\frac{1}{2}}}\frac{(2c)^{\frac{d}{2}}}{(2\pi)^{\frac{d}{2}}(t-s)^{\frac{d}{2}}}\exp(-c \frac{|y-x|^{2}}{t-s}) \notag \\
\label{neweq1.1.1}\ge &  \frac{(2c)^{\frac{d}{2}}}{(2\pi)^{\frac{d}{2}}(t-s)^{\frac{d+1}{2}}}\exp(-c \frac{|y-x|^{2}}{t-s}).
\end{align}
Suppose that $l>0$ is such that $N_{l}^{c,+}(f) <\infty$. Set $s_{1}:=s+l$. Then
\begin{align*}&\int_{s}^{s+2l}\int_{\mathbb{R}^{d}}\frac{1}{(t-s)^{\frac{d+1}{2}}}\exp(-c  \frac{|y-x|^{2}}{t-s})|f(t,y)|dydt \\
\le & N^{c,+}_{l}(f)+\int_{s+l}^{s+2l}\int_{\mathbb{R}^{d}}\frac{1}{(t-s)^{\frac{d+1}{2}}}\exp(-c  \frac{|y-x|^{2}}{t-s})|f(t,y)|dydt,
\end{align*}
and by (\ref{neweq1.1.1})
 \begin{align*}&\int_{s+l}^{s+2l}\int_{\mathbb{R}^{d}}\frac{1}{(t-s)^{\frac{d+1}{2}}}\exp(-c \frac{|y-x|^{2}}{t-s})|f(t,y)|dydt  \\
\le &\int_{\mathbb{R}^{d}}\frac{(2c)^{\frac{d}{2}}}{(2\pi)^{\frac{d}{2}}l^{\frac{d}{2}}}\exp(-c  \frac{|z-x|^{2}}{l})dz\int_{s+l}^{s+2l}\int_{\mathbb{R}^{d}}\frac{|f(t,y)|}{(t-s_{1})^{\frac{d+1}{2}}}\exp(-c  \frac{|y-z|^{2}}{t-s_{1}})dydt \\
\le & N^{c,+}_{l}(f)\int_{\mathbb{R}^{d}}\frac{(2c)^{\frac{d}{2}}}{(2\pi)^{\frac{d}{2}}l^{\frac{d}{2}}}\exp(-c \frac{|z-x|^{2}}{l})dz \\
\le & N^{c,+}_{l}(f).
\end{align*}
Therefore, we get \[
\int_{s}^{s+2l}\int_{\mathbb{R}^{d}}\frac{1}{(t-s)^{\frac{d+1}{2}}}\exp(-c \frac{|y-x|^{2}}{t-s})|f(t,y)|dydt\le 2N^{c,+}_{l}(f).\]
Similarly, we can prove for all $n \in \mathbb{N}$ \[
\int_{s}^{s+nl}\int_{\mathbb{R}^{d}}\frac{1}{(t-s)^{\frac{d+1}{2}}}\exp(-c  \frac{|y-x|^{2}}{t-s})|f(t,y)|dydt\le nN^{c,+}_{l}(f).\]
Hence the assertion follows. 
\end{proof}

The following result is \cite[Proposition 2.4]{MR1783642}. For the reader's convenience, we give a proof here.
\begin{lemma}\label{nlemma2.1.2}Suppose $f \in \CF \CK^{c}_{d-1}$ and set $f$ to be $0$ on $\mathbb{R}^{d+1} \setminus \hs$. Then for any nonnegative $\phi \in C_{0}^{\infty}(\mathbb{R}^{d+1})$ with $\int_{\mathbb{R}^{d+1}} \phi(\xi)d\xi=1$, we have
\[
  N^{c,+}_{h}(f \ast \phi)\le N^{c,+}_{h}(f).\]
Here, $f \ast \phi$ denotes the convolution of $\phi$ and $f$, that is, $f \ast \phi(\xi)=\int f(\xi-\eta)\phi(\eta)d\eta$.
\end{lemma}
\begin{proof}
For each fixed $(s,x) \in \hs$, let
\[
 g_{s,x}(t,y)=\frac{1}{(t-s)^{\frac{d+1}{2}}}\exp(-c \frac{|y-x|^{2}}{t-s}).\]

By Fubini's theorem, we get
\begin{align}
    &\int_{s}^{s+h}\int_{\mathbb{R}^{d}}\frac{1}{(t-s)^{\frac{d+1}{2}}}\exp(-c \frac{|y-x|^{2}}{t-s})|f \ast \phi|(t,y)dydt \notag \\
=&\int_{s}^{s+h}\int_{\mathbb{R}^{d}}g_{s,x}(t,y)\Big|\int_{\mathbb{R}^{d+1}}f (t-\tau,y-z)\phi(\tau,z)d\tau dz\Big|dydt \notag \\
\le & \int_{s}^{s+h}\int_{\mathbb{R}^{d}}g_{s,x}(t,y)\int_{\mathbb{R}^{d+1}}|f (t-\tau,y-z)|\phi(\tau,z)d\tau dzdydt \notag \\
\label{BMSnewnewnew2}=  & \int_{\mathbb{R}^{d+1}}\phi(\tau,z) \Big(\int_{s}^{s+h}\int_{\mathbb{R}^{d}}g_{s,x}(t,y)|f (t-\tau,y-z)|dydt\Big)d\tau dz.
\end{align}
Let $y-z=y^{\prime}, t-\tau=t^{\prime}$. Then
\begin{align}& \int_{s}^{s+h}\int_{\mathbb{R}^{d}}g_{s,x}(t,y)|f (t-\tau,y-z)|dydt \notag \\
=& \int_{s}^{s+h}\int_{\mathbb{R}^{d}}\frac{1}{(t-s)^{\frac{d+1}{2}}}\exp(-c  \frac{|y-x|^{2}}{t-s})|f(t-\tau,y-z)|dydt \notag \\
\label{BMSnewnewnew}=& \int_{s-\tau}^{s-\tau+h}\int_{\mathbb{R}^{d}}\frac{1}{(t^{\prime}-(s-\tau))^{\frac{d+1}{2}}}\exp(-c  \frac{|y^{\prime}-(x-z)|^{2}}{(t^{\prime}-(s-\tau))})|f(t^{\prime},y')|dy^{\prime}dt^{\prime} .
\end{align}

Note that $f$ is equal to $0$ on $\mathbb{R}^{d+1} \setminus \hs$. If $s-\tau\ge 0$ or $s-\tau+h\le 0$, then it follows from (\ref{BMSnewnewnew}) that
\begin{equation}\label{BMSnewnewnew1}
\int_{s}^{s+h}\int_{\mathbb{R}^{d}}g_{s,x}(t,y)|f (t-\tau,y-z)|dydt \le N^{c,+}_{h}(f).
\end{equation}
If $s-\tau<0<s-\tau+h$, then we can write $l_1:=\tau-s$, $l_2:=s-\tau+h$, and conclude from (\ref{neweq1.1.1}) and  (\ref{BMSnewnewnew}) that
\begin{align}
& \int_{s}^{s+h}\int_{\mathbb{R}^{d}}g_{s,x}(t,y)|f (t-\tau,y-z)|dydt \notag \\
= & \int_{0}^{l_2}\int_{\mathbb{R}^{d}}\frac{1}{(t^{\prime}+l_1)^{\frac{d+1}{2}}}\exp(-c  \frac{|y^{\prime}-(x-z)|^{2}}{t^{\prime}+l_1})|f(t^{\prime},y')|dy^{\prime}dt^{\prime} \notag \\
\le &\int_{\mathbb{R}^{d}}\frac{(2c)^{\frac{d}{2}}}{(2\pi)^{\frac{d}{2}}(l_1)^{\frac{d}{2}}}\exp(-c  \frac{|z'-(x-z)|^{2}}{l_1})\int_{0}^{l_2}\int_{\mathbb{R}^{d}}g_{0,z'}(t',y')|f(t',y')|dy'dt'dz' \notag\\
\le &\int_{\mathbb{R}^{d}}\frac{(2c)^{\frac{d}{2}}}{(2\pi)^{\frac{d}{2}}(l_1)^{\frac{d}{2}}}\exp(-c  \frac{|z'-(x-z)|^{2}}{l_1})\int_{0}^{h}\int_{\mathbb{R}^{d}}g_{0,z'}(t',y')|f(t',y')|  dy'dt'dz' \notag \\
\le &\int_{\mathbb{R}^{d}}\frac{(2c)^{\frac{d}{2}}}{(2\pi)^{\frac{d}{2}}(l_1)^{\frac{d}{2}}}\exp(-c  \frac{|z'-(x-z)|^{2}}{l_1})N^{c,+}_{h}(f)dz' \notag \\
\label{BMSnewnewnew3}=& N^{c,+}_{h}(f).
\end{align}

Since $\int_{\mathbb{R}^{d+1}} \phi(\tau,z)d\tau dz=1$, it follows from the inequalities (\ref{BMSnewnewnew2}), (\ref{BMSnewnewnew1}) and (\ref{BMSnewnewnew3}) that
$N^{c,+}_{h}(f \ast \phi)\le N^{c,+}_{h}(f)$.
\end{proof}

We now fix $c>0$ and suppose that $f \in \CF \CK^{c}_{d-1}$. By Definition \ref{defifk}, it is easy to see that $f$ is locally integrable. For any compact set $K \subset \hs$, we can define a finite measure
\begin{equation}\label{defiofmu}
\mu(d\xi):=\mathbf{1}_{K}(\xi) |f|(\xi)m(d\xi)
\end{equation}
on $(\hs, \mathcal{B})$, where $m$ is the Lebesgue measure on $\mathbb{R}^{d+1}$. The following lemma is straightforward.

\begin{lemma}\label{newlemma1.1.1}For each $(s,x)\in \hs$, define \[ g_{s,x}(t,y):=\frac{1}{(t-s)^{\frac{d+1}{2}}}\exp(-c \frac{|y-x|^{2}}{t-s}), \quad (t,y) \in (s,\infty)\times \Rd, \]
and $g_{s,x}(t,y):=0$ for $\ (t,y) \in [0,s] \times \Rd$. Then the family $\{g_{s,x}: (s,x)\in \hs \}$ of functions on $\hs$ is uniformly integrable with respect to the measure $\mu$, where $\mu$ is defined by (\ref{defiofmu}).
\end{lemma}
\begin{proof}Let $a>0$ and $h(a):=a^{-\frac{2}{d+1}}$. Then we have
\begin{align*}&\int_{\{g_{s,x}>a\}}g_{s,x}d\mu \\
=&\int_{\{(t,y):g_{s,x}(t,y)>a\}\cap K}\frac{1}{(t-s)^{\frac{d+1}{2}}}\exp(-c \frac{|y-x|^{2}}{t-s})|f(t,y)|dydt \\
\le & \int_{s}^{s+h(a)}\int_{\Rd}\frac{1}{(t-s)^{\frac{d+1}{2}}}\exp(-c \frac{|y-x|^{2}}{t-s})|f(t,y)|dydt\\
\le & N^{c,+}_{h(a)}(f).
\end{align*}
Since $f \in \CF \CK^{c}_{d-1}$ and $h(a)$ tends to $0$ as $a \to \infty$, it follows that
\[
\lim_{a \to \infty} \bigg ( \sup_{(s,x)\in \hs} \int_{\{g_{s,x}>a\}}g_{s,x}d\mu \bigg) =0.\]
Therefore, $\{g_{s,x}: (s,x)\in \hs \}$ is uniformly integrable with respect to the measure $\mu$. 
\end{proof}

The following proposition is an extension of \cite[Proposition 2.4(ii)]{MR1783642}.
\begin{proposition}\label{nlemma2.1.3} Given any nonnegative $\phi \in C_{0}^{\infty}(\mathbb{R}^{d+1})$ with $\int_{\mathbb{R}^{d+1}} \phi(\xi)d\xi=1$, define $\phi_{n}(\xi):=n^{d+1}\phi(n\xi)$, $\xi \in \mathbb{R}^{d+1}$. Suppose $f \in \CF \CK^{c}_{d-1}$. Then for any compact set $K \subset \hs$,
\[
   \lim_{n \to \infty}N_{h}^{c,+}(  |f\ast \phi_{n}-f| \mathbf{1}_{K} )=0.\]
\end{proposition}
\begin{proof}
Without loss of generality, we assume that the support of $\phi$ is included in the unit ball of $\mathbb{R}^{d+1}$, namely $\phi(\eta)=0$ for $|\eta| > 1$. It follows that $\phi_n(\eta)=0$ for $|\eta| > \frac{1}{n}$.

 For fixed $(s,x) \in \hs$, let  $g_{s,x}$ be as in Lemma \ref{newlemma1.1.1} and set
\[
    A:=[s,s+h]\times \mathbb{R}^{d}.\]
Then
\begin{align}
&\int_{s}^{s+h}\int_{\mathbb{R}^{d}}\frac{1}{(t-s)^{\frac{d+1}{2}}}\exp(-c \frac{|y-x|^{2}}{t-s})(|f\ast \phi_{n}-f| \mathbf{1}_{K})(t,y)dydt \notag \\
&=\int_{A}(|f\ast \phi_{n}-f|~g_{s,x})(\xi) \mathbf{1}_{K}(\xi)d\xi \notag\\
&=\int_{A}\Bigl\lvert\int_{\mathbb{R}^{d+1}}(f(\xi-\eta)-f(\xi))\phi_{n}(\eta)d\eta\Bigr\rvert g_{s,x}(\xi) \mathbf{1}_{K}(\xi)d\xi\notag \\
\label{neweq2.1.3}& \le \int_{\mathbb{R}^{d+1}}\phi_{n}(\eta) \int_{A}|f(\xi-\eta)-f(\xi)|g_{s,x}(\xi) \mathbf{1}_{K}(\xi)d\xi d\eta.
\end{align}
Set $K^{1}:=\{\xi \in \BR^{d+1}:  d(\xi,K)\le 1\}$, where $d(\xi,K):=\inf \{ |\xi-\eta|: \eta \in K \}$. By Lusin's theorem, for a given $ \delta>0$, there exist a closed set $F^{\delta} \subset K^{1}$ and a continuous function $f_{\delta}$ on $\BR^{d+1}$ with compact support such that
\[
    m(K^{1} \setminus F^{\delta})<\delta \quad \text{and} \quad f_{\delta}=f \ \text{on}  \ F^{\delta}. \]
Let $|\eta| \le 1$ and $C:=\big (K^{1}\setminus F^{\delta}\big)\cup \big ((K^{1} \setminus F^{\delta})+\eta\big )$. It is easy to see that
\[
   m(C )< 2 \delta.\]
If $\xi \in K$ and $\xi \notin F^{\delta} \cap (F^{\delta}+\eta)$, then $\xi \notin F^{\delta}$ or $\xi-\eta \notin F^{\delta}$, and thus $\xi \in \big (K^{1}\setminus F^{\delta}\big)\cup \big ((K^{1} \setminus F^{\delta})+\eta\big )$. So we get
\[
  K \subset \big (F^{\delta} \cap (F^{\delta}+\eta)\big )  \cup C.\]
Therefore,
\begin{align}
&\int_{A}|f(\xi-\eta)-f(\xi)|g_{s,x}(\xi)  \mathbf{1}_{K}(\xi)d\xi \notag\\
= & \int_{A\cap K}|f(\xi-\eta)-f(\xi)|g_{s,x}(\xi) \mathbf{1}_{K}(\xi)d\xi \notag\\
\le &\int_{A \cap  F^{\delta} \cap (F^{\delta}+\eta)}|f(\xi-\eta)-f(\xi)|g_{s,x}(\xi)d\xi+\int_{C}g_{s,x}(\xi) \mathbf{1}_{K}(\xi)|f(\xi-\eta)-f(\xi)|d\xi \notag \\
\le &\int_{A \cap  F^{\delta} \cap (F^{\delta}+\eta)}|f_{\delta}(\xi-\eta)-f_{\delta}(\xi)|g_{s,x}(\xi)d\xi+\int_{C}g_{s,x}(\xi) \mathbf{1}_{K}(\xi)|f(\xi-\eta)-f(\xi)|d\xi \notag\\
\label{neweq2.1.4}=:&  \rm{I}+\rm{II}.
\end{align}

Suppose $\epsilon>0$ is arbitrary. By Lemma \ref{newlemma1.1.1}, the family $\{g_{s,x}: (s,x)\in \hs \}$ is uniformly integrable with respect to the finite measure $\mathbf{1}_{K^{1}}(\xi) |f|(\xi)m(d\xi)$. Noting that $m(C) < 2 \delta$, we can choose $\delta$ small enough such that
\begin{align*}
   \textrm{II} =&\int_{C}g_{s,x}(\xi) \mathbf{1}_{K}(\xi)|f(\xi-\eta)-f(\xi)|d\xi \\
   \le & \int_{C}g_{s,x}(\xi) \mathbf{1}_{K}(\xi)|f(\xi)|d\xi+\int_{C-\eta}g_{s',x'}(\xi ) \mathbf{1}_{K^{1}}(\xi)|f(\xi )|d\xi  < \epsilon,
\end{align*}where $(s',x'):=(s,x)-\eta$.

Since $f_{\delta}$ is continuous and of compact support,  we can choose $n_{0}=n_{0}(\delta)$ large enough such that $|f_{\delta}(\xi-\eta)-f_{\delta}(\xi)|< \epsilon$ whenever $|\eta|\le 1/n_{0}$. If $|\eta| \le 1/n_0$, then
\begin{align*}\textrm{I}=&\int_{A \cap  F^{\delta} \cap (F^{\delta}+\eta)}|f_{\delta}(\xi-\eta)-f_{\delta}(\xi)|g_{s,x}(\xi)d\xi \\
\le & \epsilon \int_{A \cap  F^{\delta} \cap (F^{\delta}+\eta)}g_{s,x}(\xi)d\xi \le \epsilon \int_{A}g_{s,x}(\xi)d\xi .\end{align*}
Set $M:=\int_{A}g_{s,x}(\xi)d\xi$. By (\ref{neweq2.1.4}) and noting that $\phi_n(\eta)=0$ for $|\eta| > 1/n$, we have
\[
 0 \le \int_{\mathbb{R}^{d+1}}\phi_{n}(\eta) \int_{A}|f(\xi-\eta)-f(\xi)|g_{s,x}(\xi)  \mathbf{1}_{K}(\xi)d\xi d\eta<(M+1)\epsilon \]
for $n \ge n_0$.

 Using (\ref{neweq2.1.3}) and noting that the constant $M$ and the choice of $\delta$ are independent of $(s,x)$, we get
\[
    \sup_{(s,x)\in \hs}\int_{s}^{s+h}\int_{\mathbb{R}^{d}}\frac{1}{(t-s)^{\frac{d+1}{2}}}\exp(-c  \frac{|y-x|^{2}}{t-s})(|f\ast \phi_{n}-f| \mathbf{1}_{K})dydt \to 0 \]
as $n \to \infty$. 
\end{proof}

%%%%%% Section 3: Some gradient estimates for R^{\lambda}
\section{Some Gradient Estimates for the Resolvent of Brownian Motion}In this section we derive some gradient estimates for the space-time resolvent $R^{\lambda}$ ($\lambda>0$) of Brownian motion. Recall that the transition density function $p(s,x;t,y)$ of Brownian motion is given by
\[
     p(s,x;t,y)=\frac{1}{(2\pi)^{\frac{d}{2}}(t-s)^{\frac{d}{2}}}\exp\Big(-\frac{|x-y|^{2}}{2(t-s)}\Big).\]

Throughout this section we fix a positive constant $\alpha$ with $\alpha<1/2$. It is easy to verify that there exists a constant $C_{1}>1$, depending on $\alpha$, such that for all $0 \le s <t$ and $x,y \in \Rd$,
\begin{equation}\label{eq2.2.2}
|\nabla_{x}p(s,x;t,y)| \le \frac{C_{1}}{(t-s)^{\frac{d+1}{2}}}\exp\Big(-\alpha  \frac{|x-y|^{2}}{t-s}\Big).
\end{equation}

For any $\lambda>0$ let $R^{\lambda}$ be the space-time resolvent of Brownian motion, namely
\begin{equation}\label{resbm}
R^{\lambda}f(s,x):=\int_{s}^{\infty}e^{-\lambda(t-s)}\int_{\mathbb{R}^{d}}p(s,x;t,y)f(t,y)dydt, \quad (s,x) \in \hs, \end{equation}
where $f$ is a bounded and measurable function on $\hs$.

\begin{lemma}\label{lemma2.2.2.1}Let $\alpha<1/2$ be a positive constant. If $f \in  \CF \CK^{\alpha}_{d-1}$ and $\textup{supp}(f) \subset [s_{1},s_{1}+h]\times \mathbb{R}^{d}$ for some $s_{1}\ge 0$ and $h>0$, then
\[
    |\nabla R^{\lambda}f(s,x)|\le C_{1} N_{h}^{\alpha,+}(f) \]
for all $(s,x) \in \hs$, where $C_{1}$ is the constant appearing in (\ref{eq2.2.2}).
\end{lemma}
\begin{proof}
We have
\begin{align}&|\nabla R^{\lambda}f(s,x)| \notag\\
=&|\nabla \int_{s}^{\infty}e^{-\lambda(t-s)}dt\int_{\mathbb{R}^{d}}p(s,x;t,y)f(t,y)dy| \notag\\
=&|\int_{s}^{\infty}e^{-\lambda(t-s)}dt\int_{\mathbb{R}^{d}}\nabla_{x} p(s,x;t,y)f(t,y)dy| \notag\\
\label{neweq3.1.0}\le &\int_{s}^{\infty}e^{-\lambda(t-s)}\int_{\mathbb{R}^{d}}C_{1}\frac{1}{(t-s)^{\frac{d+1}{2}}}\exp(-\alpha  \frac{|x-y|^{2}}{t-s})|f(t,y)|dydt \\
\le &\int_{s}^{\infty}\int_{\mathbb{R}^{d}}C_{1}\frac{1}{(t-s)^{\frac{d+1}{2}}}\exp(-\alpha  \frac{|x-y|^{2}}{t-s})|f(t,y)|dydt. \notag
\end{align}
If $s \ge s_{1}$, then
\begin{align}&|\nabla R^{\lambda}f(s,x)| \notag \\
\label{graesti1} \le  &\int_{s}^{s+h}\int_{\mathbb{R}^{d}}C_{1}\frac{1}{(t-s)^{\frac{d+1}{2}}}\exp(-\alpha  \frac{|x-y|^{2}}{t-s})|f(t,y)|dydt \\
\le & C_{1} N^{\alpha,+}_{h}(f). \notag
\end{align}
If $s<s_{1}$, then
\begin{align}&|\nabla R^{\lambda}f(s,x)|\notag\\
\le &\int_{s_{1}}^{s_{1}+h}\int_{\mathbb{R}^{d}}C_{1}\frac{1}{(t-s)^{\frac{d+1}{2}}}\exp(-\alpha  \frac{|x-y|^{2}}{t-s})|f(t,y)|dydt \notag\\
\label{eq2.2.3}\le &\int_{\mathbb{R}^{d}}\frac{C_{1}(2\alpha)^{\frac{d}{2}}}{[2\pi(s_{1}-s)]^{\frac{d}{2}}}\exp(-\alpha \frac{|z-x|^{2}}{s_{1}-s})dz\int_{s_{1}}^{s_{1}+h}\int_{\mathbb{R}^{d}}\frac{|f(t,y)|}{(t-s_{1})^{\frac{d+1}{2}}}\exp(-\alpha  \frac{|y-z|^{2}}{t-s_{1}})dydt \\
\le & C_{1}N^{\alpha,+}_{h}(f)\int_{\mathbb{R}^{d}}\frac{(2\alpha)^{\frac{d}{2}}}{(2\pi)^{\frac{d}{2}}(s_{1}-s)^{\frac{d}{2}}}\exp(-\alpha \frac{|z-x|^{2}}{s_{1}-s})dz  \notag\\
\le &C_{1} N^{\alpha,+}_{h}(f). \notag
\end{align}
In fact, to obtain (\ref{eq2.2.3}), we only need to apply the inequality (\ref{neweq1.1.1}). The lemma is proved. 
\end{proof}

Similar to the above lemma, we have the following  estimate for $R^{\lambda}$.
\begin{lemma}\label{lemma2.2.2.3}
Suppose $0<\alpha<1/2$. If $f \in  \CF \CK^{\alpha}_{d-1}$ and $\textup{supp}(f) \subset [s_{1},s_{1}+h]\times \mathbb{R}^{d}$ for some $s_{1}\ge 0$ and $0<h<1$, then
\[
    |R^{\lambda}f|\le  N_{h}^{\alpha,+}(f). \]
\end{lemma}
\begin{proof}The proof is similar to Lemma \ref{lemma2.2.2.1}. We only need to note that if $0<t-s<1$, then
\[
p(s,x;t,y) \le \frac{1}{(t-s)^{\frac{d+1}{2}}}\exp(-\alpha \frac{|x-y|^{2}}{t-s})   \]
for all $x,y \in \Rd$.  
\end{proof}

The following lemma is a simple consequence of the inequality (\ref{neweq3.1.0}).
\begin{lemma}\label{lemma2.2.2.2}For each $\lambda >0$, there exists a constant $C_{\lambda}>0$ such that
\[
    |\nabla R^{\lambda}g(s,x)|\le C_{\lambda}  \|g\|_{\infty}\]
for all $(s,x) \in \hs$ and  $g \in \mathcal{B}_{b}([0,\infty)\times \mathbb{R}^{d})$.
\end{lemma}

%%%%%%%%   Section 4: Existence and uniqueness of weak solution to (1)

\section{Existence and Uniqueness of Weak Solutions: Local Case}
Instead of dealing with weak solutions to the SDE (\ref{SDE1}) directly, we use the equivalent martingale-problem formulation. Namely, we shall prove that the martingale problem for the generator
\[
 L_{t}=\frac{1}{2}\triangle + b(t,\cdot) \cdot \nabla
\] is well-posed.

One of the advantages to use the martingale-problem approach is that the martingale problem can be reduced to local considerations. In other words, we can first assume that the support of the drift $b$ is compact and then construct a ``local solution" to the SDE (\ref{SDE1}). Then we use the ``glueing argument" to get a global solution.

Throughout this section we assume the following assumption holds, that is, we confine ourselves to the local case. The general case will be discussed in the next section.
\begin{assumption}\label{assumption2.2.3} The drift $b$ satisfies $|b| \in \CF \CK_{d-1}^{\alpha}$ for some $\alpha \in (0, 1/2)$. Furthermore, there exists $(s_{1},x_{1}) \in \hs$ such that
\[
\textup{supp} \big(|b|\big ) \subset [s_{1},s_{1}+\epsilon _{1}]\times \big \{x \in \Rd: |x-x_{1} |\le 1 \big \}, \]
and \begin{equation}\label{neweqass2}
N^{\alpha,+}_{2\epsilon_{1}}(|b|)<\frac{1}{2\kappa C_{1}},\end{equation}
where $\kappa:=d^{3/2}$, $\epsilon_{1}< 1/2$ and the constant $C_{1}$ is taken from (\ref{eq2.2.2}).
 \end{assumption}

We first consider smooth approximations of the singular drift $b$. Given a nonnegative function $\phi \in C_{0}^{\infty}(\mathbb{R}^{d+1})$ with $\int_{\mathbb{R}^{d+1}} \phi(\xi)d\xi=1$ and \begin{equation}\label{supportofphi}
\textrm{supp}(\phi) \subset \{ \xi \in \mathbb{R}^{d+1}: |\xi| \le \frac{\epsilon_1}{2}  \},
\end{equation}
let \[
\phi_{n}(\xi)=n^{(d+1)}\phi(n\xi)\]and define \[b_{n}:=b \ast \phi_{n}=(b^{1} \ast \phi_{n}, \cdots, b^{d} \ast \phi_{n}).\]

\begin{remark}By Lemma \ref{nlemma2.1.2}, it is easily seen that
\begin{align}
N^{\alpha,+}_{h}(|b_{n}|) \le & \sqrt{d}\sum_{i=1}^{d}N^{\alpha,+}_{h}(|b^i_{n}|) \le \sqrt{d}\sum_{i=1}^{d}N^{\alpha,+}_{h}(|b^i|) \notag \\
\label{neq2.2.3.1} \le & d^{\frac{3}{2}} N_{h}^{\alpha,+}(|b|)=\kappa N_{h}^{\alpha,+}(|b|)
\end{align}
for all $h>0$. Furthermore, it follows from (\ref{supportofphi}) and Assumption \ref{assumption2.2.3} that
\begin{equation}\label{supportofbn}
\textrm{supp}(b_{n}) \subset [s_{1}-\frac{\epsilon _{1}}{2},s_{1}+\frac{3\epsilon _{1}}{2}]\times K,
\end{equation}
where $K\subset\Rd$ is compact.
\end{remark}

\begin{remark}Since both $b_{n}$ and $b$ have compact support, by Proposition \ref{nlemma2.1.3}, we know that
 \begin{equation}\label{neq2.2.3}
   \lim_{n \to \infty}N_{h}^{\alpha,+}(|b_{n}-b|)=0
 \end{equation}
 for every $h>0$.
 \end{remark}

Since $b_{n}$ is smooth and has compact support, for each starting point $(s,x) \in \hs$, there exists a unique probability measure $\mathbf{P}^{s,x}_{n}$ on $\big(\Omega=C([0,\infty);\mathbb{R}^{d}),\mathcal{M}\big)$ that solves the martingale problem for the generator
\[
\frac{1}{2}\triangle+b_{n}(t,\cdot) \cdot \nabla. \]
For any $\lambda>0$ and any bounded measurable function $f$ on $\hs$, define
\[
S^{\lambda}_{n}f(s,x):=\mathbf{E}^{s,x}_{n}\Big[\int_{s}^{\infty}e^{-\lambda(t-s)}f(t,X_{t})dt\Big], \quad (s,x) \in \hs, \]
where $\mathbf{E}^{s,x}_{n}[\cdot]$ means taking expectation with respect to the measure $\mathbf{P}^{s,x}_{n}$ on the path space $\big(\Omega=C([0,\infty);\mathbb{R}^{d}),\mathcal{M}\big)$.

Recall that $R^{\lambda}$ is the space-time resolvent for Brownian motion defined in (\ref{resbm}). For any $f \in \mathcal{B}_{b}([0,\infty)\times \mathbb{R}^{d})$, since the first order partial derivatives of $R^{\lambda}f$ exist and are continuous,  we can define the operator $BR^{\lambda}$ as follows
\begin{equation}\label{defiofBR}
    BR^{\lambda}f(s,x):=\sum_{i=1}^{d}b^{i}(s,x) \frac{\partial R^{\lambda}f}{\partial x_{i}}(s,x), \quad (s,x) \in \hs. \end{equation}
Similarly, for $f \in \mathcal{B}_{b}([0,\infty)\times \mathbb{R}^{d})$ define
\[
    B_{n}R^{\lambda}f(s,x):=\sum_{i=1}^{d}b^{i}_{n}(s,x) \frac{\partial R^{\lambda}f}{\partial x_{i}}(s,x), \quad (s,x) \in \hs . \]

\begin{lemma}\label{newlemmasn}If $g \in \mathcal{B}_{b}([0,\infty)\times \mathbb{R}^{d})$, then
\begin{equation}\label{eqressn}
    S_{n}^{\lambda}g=\sum_{k=0}^{\infty}R^{\lambda}(B_{n}R^{\lambda})^{k}g,\end{equation}
where the series on the right-hand side of (\ref{eqressn}) converges uniformly on $\hs$.
\end{lemma}

\begin{proof}Since $b_{n}$ is smooth and has compact support, Brownian motion with such a drift $b_{n}$ has a transition density function $q_{n}(s,x;t,y)$. Recall that $p(s,x;t,y)$ is the transition density function of Brownian motion. Then by Duhamel's formula (see \cite[p.~388]{MR1457736}), we have
\begin{equation}\label{neweqduhamel}
    q_{n}(s,x;t,y)=p(s,x;t,y)+\int_{s}^{t}\int_{\mathbb{R}^{d}}q_{n}(s,x;\tau,z)b_{n}(\tau,z)\cdot\nabla_{z} p(\tau,z;t,y)dzd\tau .\end{equation}
For a detailed proof of (\ref{neweqduhamel}), the reader is referred to \cite[p.~647]{MR2247841}.

By (\ref{neweqduhamel}), we can calculate the difference between $S_{n}^{\lambda}$ and $R^{\lambda}$. More precisely, if $f$ is bounded and measurable, then
\begin{align*}
& S_{n}^{\lambda}f(s,x)-R^{\lambda}f(s,x) \\
=& \int_{s}^{\infty}\int_{\mathbb{R}^{d}}e^{-\lambda(t-s)}q_{n}(s,x;t,y)f(t,y)dydt-\int_{s}^{\infty}\int_{\mathbb{R}^{d}}e^{-\lambda(t-s)}p(s,x;t,y)f(t,y)dydt \\
=& \int_{s}^{\infty}\int_{\mathbb{R}^{d}}e^{-\lambda(t-s)}(q_{n}(s,x;t,y)-p(s,x;t,y))f(t,y)dydt \\
=& \int_{s}^{\infty}\int_{\mathbb{R}^{d}}e^{-\lambda(t-s)}f(t,y)(\int_{s}^{t}\int_{\mathbb{R}^{d}}q_{n}(s,x;\tau,z)b_{n}(\tau,z)\cdot \nabla_{z}p(\tau,z;t,y)dzd\tau ) dydt.
\end{align*}
Since $b_n$ is of compact support, there exists $T>s$ such that $\textrm{supp}(b_n) \subset [0,T]\times \Rd$. According to \cite[Theorem A]{MR1457736}, we can find constants $C, \beta>0$ such that
\begin{equation}\label{gaussianbfqn}
|q_n(s,x;\tau,z)| \le C(\tau-s)^{-\frac{d}{2}}\exp(-\beta  \frac{|x-z|^{2}}{\tau-s}), \quad 0 \le s,\tau \le T,\ x,z \in \Rd.
\end{equation}
Here we can obviously choose $\beta <\alpha$. It follows from (\ref{gaussianbfqn}), (\ref{eq2.2.2}) and \cite[Lemma~ 3.1(a)]{MR1457736} that
\begin{align*}
& \int_{s}^{t}\int_{\mathbb{R}^{d}}|q_{n}(s,x;\tau,z)b_{n}(\tau,z)\cdot \nabla_{z}p(\tau,z;t,y)|dzd\tau \\
\le & \int_{s}^{t\wedge T}\int_{\mathbb{R}^{d}}|q_{n}(s,x;\tau,z)b_{n}(\tau,z)\cdot \nabla_{z}p(\tau,z;t,y)|dzd\tau \\
\le &C'(t-s)^{-\frac{d}{2}}\exp(-\beta  \frac{|x-y|^{2}}{t-s}),
\end{align*}
where $C'>0$ is a constant. Therefore, we can apply Fubini's theorem to get
\begin{align}
& S_{n}^{\lambda}f(s,x)-R^{\lambda}f(s,x) \notag \\
\label{neweqcompare1}=& \int_{s}^{\infty}\int_{\mathbb{R}^{d}}q_{n}(s,x;\tau,z)dzd\tau(\int_{\tau}^{\infty}\int_{\mathbb{R}^{d}} e^{-\lambda(t-s)}f(t,y)  b_{n}(\tau,z)\cdot\nabla_{z}p(\tau,z;t,y)dydt). \end{align}
For $\tau \ge s$ and $z \in  \Rd$, we have
\begin{align}
&\int_{\tau}^{\infty}\int_{\mathbb{R}^{d}} e^{-\lambda(t-s)}f(t,y)  b_{n}(\tau,z)\cdot\nabla_{z}p(\tau,z;t,y)dydt \notag\\
=& e^{-\lambda(\tau-s)}b_{n}(\tau,z)\cdot \nabla_{z}\big(\int_{\tau}^{\infty}\int_{\mathbb{R}^{d}} e^{-\lambda(t-\tau)}p(\tau,z;t,y)f(t,y)dydt\big) \notag\\
\label{neweqcompare2}=& e^{-\lambda(\tau-s)}B_{n}R^{\lambda}f(\tau,z).
\end{align}
Substituting (\ref{neweqcompare2}) in (\ref{neweqcompare1}), we obtain
\begin{align*}
S_{n}^{\lambda}f(s,x)-R^{\lambda}f(s,x)=&\int_{s}^{\infty}\int_{\mathbb{R}^{d}}e^{-\lambda(\tau-s)}q_{n}(s,x;\tau,z)B_{n}R^{\lambda}f(\tau,z)dzd\tau \\
=& S_{n}^{\lambda}B_{n}R^{\lambda}f(s,x).
\end{align*}
Thus, we have shown that
\[
 S_{n}^{\lambda}f-R^{\lambda}f=S_{n}^{\lambda}B_{n}R^{\lambda}f, \quad f \in \CB_b(\hs). \]
For any bounded measurable function $g$ on $\hs$, taking $f=B_{n}R^{\lambda}g$ in the above formula, we get
\[
     S_{n}^{\lambda}B_{n}R^{\lambda}g-R^{\lambda}B_{n}R^{\lambda}g=S_{n}^{\lambda}B_{n}R^{\lambda}B_{n}R^{\lambda}g, \]
and thus
\begin{align*}
    S_{n}^{\lambda}g=& R^{\lambda}g+S_{n}^{\lambda}B_{n}R^{\lambda}g \\
    =& R^{\lambda}g+R^{\lambda}B_{n}R^{\lambda}g+S_{n}^{\lambda}(B_{n}R^{\lambda})^{2}g.
\end{align*}
Similarly, after $i$ steps, we obtain

\begin{equation}\label{eqsnforn}
 S_{n}^{\lambda}g=\sum_{k=0}^{i}R^{\lambda}(B_{n}R^{\lambda})^{k}g+S_{n}^{\lambda}(B_{n}R^{\lambda})^{i+1}g.
\end{equation}
In order to show that the last term on the right-hand side of (\ref{eqsnforn}) converges to 0 as $i \to \infty$, we first need to prove the following claim.

\begin{claim}  We have $(B_{n}R^{\lambda})^{k}g \in \CF \CK^{\alpha}_{d-1}$
and
\begin{equation}\label{neq2.2.2.20}
    N^{\alpha,+}_{2\epsilon_{1}}((B_{n}R^{\lambda})^{k}g)\le C_{\lambda}\|g\|_{\infty}(\frac{1}{2})^{k}\end{equation}
for all $k \in \BN$.
\end{claim}
By (\ref{neweqass2}) and (\ref{neq2.2.3.1}), we have
\begin{equation}\label{eqclaimk1}
N^{\alpha,+}_{2\epsilon_{1}}(|b_{n}|)\le \kappa N^{\alpha,+}_{2\epsilon_{1}}(|b|)\le \frac{1}{2C_{1}}.\end{equation}
When $k=1$, by Lemma \ref{lemma2.2.2.2},
\begin{equation}\label{eqclaimk2} |B_{n}R^{\lambda}g|\le |b_{n}||\nabla R^{\lambda}g|
\le |b_{n}|C_{\lambda}\|g\|_{\infty}. \end{equation}
Thus $B_{n}R^{\lambda}g \in \CF \CK^{\alpha}_{d-1}$. Since the constant $C_1$ appearing in (\ref{eq2.2.2}) is greater than 1, by (\ref{eqclaimk1}) and (\ref{eqclaimk2}), we have
\[
    N^{\alpha,+}_{2\epsilon_{1}}(B_{n}R^{\lambda}g)\le C_{\lambda}\|g\|_{\infty} N^{\alpha,+}_{2\epsilon_{1}}(|b_{n}|)\le \frac{1}{2} C_{\lambda}\|g\|_{\infty}.\]
Suppose that the claim is true for $k$. It follows from (\ref{supportofbn}) that
\begin{equation}\label{supportofBn}
 \textrm{supp}((B_{n}R^{\lambda})^{k}g) \subset [s_{1}-\frac{\epsilon _{1}}{2},s_{1}+\frac{3\epsilon _{1}}{2}]\times K.
 \end{equation}
By Lemma \ref{lemma2.2.2.1},
 \begin{align*}
     |(B_{n}R^{\lambda})^{k+1}g|\le & |b_{n}| |\nabla R^{\lambda}(B_{n}R^{\lambda})^{k}g| \\
     \le & C_{1} |b_{n}|  N^{\alpha,+}_{2\epsilon_{1}}((B_{n}R^{\lambda})^{k}g)\\
    \le & C_{1} C_{\lambda}\|g\|_{\infty}|b_{n}|  (\frac{1}{2})^{k}.
 \end{align*}
Therefore,
\[
    N^{\alpha,+}_{2\epsilon_{1}}((B_{n}R^{\lambda})^{k+1}g)\le C_{1} C_{\lambda}\|g\|_{\infty} (\frac{1}{2})^{k} \frac{1}{2C_{1}} \le C_{\lambda}\|g\|_{\infty} (\frac{1}{2})^{k+1},\] and the claim is proved.

Noting (\ref{supportofBn}) and using Lemma \ref{lemma2.2.2.1} again,
\begin{align} |S_{n}^{\lambda}(B_{n}R^{\lambda})^{k+1}g| =&|S_{n}^{\lambda}(b_{n} \cdot \nabla R^{\lambda}(B_{n}R^{\lambda})^{k}g)| \notag\\
\le & S_{n}^{\lambda}(|b_{n}||\nabla R^{\lambda}(B_{n}R^{\lambda})^{k}g|) \notag\\
\le & C_{1}N^{\alpha,+}_{2\epsilon_{1}}((B_{n}R^{\lambda})^{k}g)S_{n}^{\lambda}(|b_{n}|)\notag\\
\label{eqsnbound}\le & C_{1}C_{\lambda}\|g\|_{\infty}(\frac{1}{2})^{k}S_{n}^{\lambda}(|b_{n}|).
\end{align}
Since $|b_n|$ is smooth and has compact support, the term $S_{n}^{\lambda}(|b_{n}|)$ is bounded and thus
\begin{equation}\label{snkcon}
\lim_{k \to \infty}|S_{n}^{\lambda}(B_{n}R^{\lambda})^{k+1}g| = 0.
\end{equation}
Similarly to (\ref{eqsnbound}), we get
\[
   |R^{\lambda}(B_{n}R^{\lambda})^{k}g|\le C_{1}C_{\lambda}\|g\|_{\infty}(\frac{1}{2})^{k-1}R^{\lambda}(|b_{n}|).\]
Noting that $\epsilon_{1}< 1/2$ and using Lemma \ref{lemma2.2.2.3}, (\ref{supportofbn}) and (\ref{eqclaimk1}),
\begin{equation}\label{rlambdabn}
    |R^{\lambda}(|b_{n}|)| \le N^{\alpha,+}_{2\epsilon_{1}}(|b_{n}|) \le  \frac{1}{2C_{1}}.\end{equation}
Therefore, \begin{equation}\label{neweqrn}
   |R^{\lambda}(B_{n}R^{\lambda})^{k}g| \le C_{\lambda}\|g\|_{\infty}(\frac{1}{2})^{k}. \end{equation}
Combining (\ref{eqsnforn}), (\ref{snkcon}) and (\ref{neweqrn}) yields our result. 
\end{proof}

\begin{remark}If we check the proof of Lemma \ref{newlemmasn}, the essential conditions that we used to ensure the uniform convergence of the series in (\ref{eqressn}) are (\ref{supportofbn}) and
\[
    N^{\alpha,+}_{2\epsilon_{1}}\big(|b_{n}|\big)<\frac{1}{2C_{1}}.\]
By Assumption \ref{assumption2.2.3}, we can actually replace $b_n$ by $b$ in the above arguments. Therefore, given any bounded measurable function $g$, we can define $S^{\lambda}g$ as follows:
\begin{equation}\label{neweqdefiofslambda}
    S^{\lambda}g:=\sum_{k=0}^{\infty}R^{\lambda}(BR^{\lambda})^{k}g. \end{equation} Moreover, for each term of the above series, we have
\begin{equation}\label{neweqnormofBRk}
    N^{\alpha,+}_{2\epsilon_{1}}((BR^{\lambda})^{k}g)\le C_{\lambda}\|g\|_{\infty} (\frac{1}{2})^{k}
\end{equation}
and
\begin{equation}\label{neweqr}
     |R^{\lambda}(BR^{\lambda})^{k}g|\le C_{\lambda}\|g\|_{\infty}(\frac{1}{2})^{k}.\end{equation}
\end{remark}

We now prove that $S^{\lambda}$ is the limit of $S^{\lambda}_{n}$ as $n \to \infty$.
\begin{lemma}\label{newlemmasc}For each bounded measurable function $g$ on $\hs$, $S_{n}^{\lambda}g$ converges to $S^{\lambda}g$ uniformly on $ [0,\infty)\times \mathbb{R}^{d}$ as $n \to \infty$.
\end{lemma}
\begin{proof}We first show that
\[
    \lim_{n \to \infty}R^{\lambda}B_{n}R^{\lambda}g(s,x)=R^{\lambda}BR^{\lambda}g(s,x),\]
where the convergence is uniform with respect to $(s,x) \in [0,\infty)\times \mathbb{R}^{d}$. In fact, by Lemma \ref{lemma2.2.2.3} and Lemma \ref{lemma2.2.2.2},
\begin{align*}& |R^{\lambda}B_{n}R^{\lambda}g(s,x)-R^{\lambda}BR^{\lambda}g(s,x)|\\
=& |R^{\lambda}(B_{n}-B)R^{\lambda}g(s,x)|\\
\le & C_{\lambda}\|g\|_{\infty} R^{\lambda}(|b_{n}-b|)(s,x)\\
\le & C_{\lambda}\|g\|_{\infty}N^{\alpha,+}_{2\epsilon_{1}}(|b-b_{n}|).
\end{align*}
 By (\ref{neq2.2.3}), we have
\[
\lim_{n \to \infty}R^{\lambda}B_{n}R^{\lambda}g(s,x)=R^{\lambda}BR^{\lambda}g(s,x), \quad \text{uniformly in $(s,x) \in \hs$}.\]
We next show that for each $k \in \BN$,

\begin{equation}\label{conforrbn}
    \lim_{n \to \infty}R^{\lambda}(B_{n}R^{\lambda})^{k}g(s,x)=R^{\lambda}(BR^{\lambda})^{k}g(s,x), \quad \text{uniformly in $(s,x) \in [0,\infty)\times \mathbb{R}^{d}$},
\end{equation}
and
\begin{equation}\label{conforrbn2}
    \lim_{n \to \infty}N^{\alpha,+}_{2\epsilon_{1}}[(B_{n}R^{\lambda})^{k}g-(BR^{\lambda})^{k}g]=0.\end{equation}
The case $k=1$ has already been treated. Suppose the assertions (\ref{conforrbn}) and (\ref{conforrbn2}) are true for $k$.
By Lemma \ref{lemma2.2.2.1},
\begin{align*}&|(B_{n}R^{\lambda})^{k+1}g-(BR^{\lambda})^{k+1}g|\\
=& |(B_{n}R^{\lambda})^{k+1}g-b_{n}\cdot\nabla R^{\lambda}(BR^{\lambda})^{k}g+b_{n}\cdot\nabla R^{\lambda}(BR^{\lambda})^{k}g-(BR^{\lambda})^{k+1}g|\\
\le & |b_{n}||\nabla R^{\lambda}(B_{n}R^{\lambda})^{k}g-\nabla R^{\lambda}(BR^{\lambda})^{k}g|+|b_{n}-b||\nabla R^{\lambda}(BR^{\lambda})^{k}g|\\
\le & C_{1} |b_{n}| N^{\alpha,+}_{2\epsilon_{1}}((B_{n}R^{\lambda})^{k}g-(BR^{\lambda})^{k}g)+C_{1} |b_{n}-b|N^{\alpha,+}_{2\epsilon_{1}}((BR^{\lambda})^{k}g).
\end{align*}
Similarly to (\ref{neq2.2.2.20}), we obtain
   \[ N^{\alpha,+}_{2\epsilon_{1}}((BR^{\lambda})^{k}g)\le C_{\lambda}\|g\|_{\infty}(\frac{1}{2})^{k}.\]
So, we get
\[
    \lim_{n \to \infty}N^{\alpha,+}_{2\epsilon_{1}}[(B_{n}R^{\lambda})^{k+1}g-(BR^{\lambda})^{k+1}g]=0.\]
As $n \to \infty$, by Lemma \ref{lemma2.2.2.3},
\begin{align*}& |R^{\lambda}(B_{n}R^{\lambda})^{k+1}g(s,x)-R^{\lambda}(BR^{\lambda})^{k+1}g(s,x)|\\
\le & N^{\alpha,+}_{2\epsilon_{1}}[(B_{n}R^{\lambda})^{k+1}g-(BR^{\lambda})^{k+1}g] \to 0,
\end{align*}and the convergence is uniform in $(s,x) \in [0,\infty)\times \mathbb{R}^{d}$.

Now,  the assertion follows from (\ref{neweqrn}), (\ref{neweqr}) and (\ref{conforrbn}). 
\end{proof}

\begin{lemma}\label{lemma2.2.3.50}As $\lambda \to \infty$, $S_{n}^{\lambda}|b_{n}|(s,x)$ converges to $0$ uniformly in $(s,x) \in \hs$.  Moreover, the convergence rate is independent of $n$.
\end{lemma}

\begin{proof}
Let $m,n \in \BN$. Since $|b_{m}|$ is bounded, by Lemma \ref{newlemmasn},
\[
S_{n}^{\lambda}|b_{m}|=\sum_{k=0}^{\infty}R^{\lambda}(B_{n}R^{\lambda})^{k}|b_{m}|.
\]
Similarly to (\ref{neq2.2.2.20}), it follows from Lemma \ref{lemma2.2.2.1}, (\ref{eqclaimk1}) and (\ref{supportofBn}) that
\begin{align*}
N^{\alpha,+}_{2\epsilon_{1}}((B_{n}R^{\lambda})^{k}|b_{m}|) = & N^{\alpha,+}_{2\epsilon_{1}}\big[b_{n} \cdot \nabla \big(R^{\lambda}(B_{n}R^{\lambda})^{k-1}|b_{m}|\big)\big] \\
      \le & N^{\alpha,+}_{2\epsilon_{1}}[C_1|b_{n}|N^{\alpha,+}_{2\epsilon_{1}}((B_{n}R^{\lambda})^{k-1}|b_{m}|)]  \\
      \le & C_1 N^{\alpha,+}_{2\epsilon_{1}}(|b_{n}|) N^{\alpha,+}_{2\epsilon_{1}}((B_{n}R^{\lambda})^{k-1}|b_{m}|)  \\
       \le & \frac{1}{2} N^{\alpha,+}_{2\epsilon_{1}}((B_{n}R^{\lambda})^{k-1}|b_{m}|)  \le  \cdots   \\
      \le & (\frac{1}{2})^{k}N^{\alpha,+}_{2\epsilon_{1}}(|b_{m}|).
\end{align*}
By Lemma \ref{lemma2.2.2.1} and (\ref{eqclaimk1}), we obtain
\begin{align}
    S_{n}^{\lambda}|b_{m}|\le & R^{\lambda} (|b_{m}|) + \sum_{k=1}^{\infty}R^{\lambda} \big(\big|b_{n}\big| \big|\nabla \big(R^{\lambda}(B_{n}R^{\lambda})^{k-1}|b_{m}|\big)\big|\big) \notag \\
\le & R^{\lambda} (|b_{m}|)+\sum_{k=1}^{\infty}C_{1}    N^{\alpha,+}_{2\epsilon_{1}}[(B_{n}R^{\lambda})^{k-1}|b_{m}|] R^{\lambda} (|b_{n}|)\notag \\
\label{estiforsnm}\le & R^{\lambda} (|b_{m}|)+\sum_{k=1}^{\infty}  C_{1}  (\frac{1}{2})^{k-1} N^{\alpha,+}_{2\epsilon_{1}}(|b_{m}|)  R^{\lambda} (|b_{n}|)\\
\label{estiforsnm2}\le & R^{\lambda} (|b_{m}|)+\sum_{k=1}^{\infty}    (\frac{1}{2})^{k} R^{\lambda} (|b_{n}|)  \le   R^{\lambda}|b_{m}|+R^{\lambda}|b_{n}|.
    \end{align}

For any given $ \epsilon >0$, we can find a $\delta \in (0,1)$ such that
\[
     N^{\alpha,+}_{\delta}(|b|)\le \frac{\epsilon}{4\kappa }.\]
By (\ref{neq2.2.3.1}), (\ref{estiforsnm2}) and noting that $\textrm{supp}(b_{n}) \subset [s_{1}-\epsilon/2,s_{1}+3\epsilon/2]\times K$ and $\epsilon _{1}<1/2$, we have
 \begin{align*}  & S_{n}^{\lambda}|b_{n}|(s,x) \le    2  R^{\lambda}|b_{n}|(s,x) \\
 \le & 2\int_{s}^{\infty}\int_{\mathbb{R}^{d}}e^{-\lambda(t-s)}p(s,x;t,y)|b_{n}|(t,y)dydt  \\
 \le & 2\Big(\int_{s}^{s+\delta}\int_{\mathbb{R}^{d}}\frac{1}{(t-s)^{\frac{d+1}{2}}}\exp(-\alpha  \frac{|x-y|^{2}}{t-s})|b_{n}|(t,y)dydt \\
 & \qquad \qquad \qquad \qquad \qquad +\int_{s+\delta}^{\infty}\int_{\mathbb{R}^{d}}e^{-\lambda \delta}p(s,x;t,y)|b_{n}|(t,y)dydt \Big )\\
 \le & 2   N^{\alpha,+}_{\delta}(|b_{n}|) + 2 e^{-\lambda \delta} \int_{s}^{\infty}\int_{\mathbb{R}^{d}}p(s,x;t,y)|b_{n}|(t,y)dydt.
\end{align*}
By the same arguments that we used to establish (\ref{graesti1}) and (\ref{eq2.2.3}), we obtain
\[
\int_{s}^{\infty}\int_{\mathbb{R}^{d}}p(s,x;t,y)|b_{n}|(t,y)dydt \le N^{\alpha,+}_{2\epsilon_{1}}(|b_{n}|).
\]
Therefore,
\begin{align*}
 S_{n}^{\lambda}|b_{n}|(s,x)
 \le & 2   N^{\alpha,+}_{\delta}(|b_{n}|)+ 2e^{-\lambda \delta}N^{\alpha,+}_{2\epsilon_{1}}(|b_{n}|)\\
 \le & 2\kappa  N^{\alpha,+}_{\delta}(|b|)+ 2\kappa e^{-\lambda \delta}N^{\alpha,+}_{2\epsilon_{1}}(|b|) \\
 \le & \frac{\epsilon}{2}+\frac{1}{C_1}e^{-\lambda \delta}.
 \end{align*}
So we can find a sufficiently large $\lambda_{0}$ such that if $\lambda > \lambda_{0}$, then $S_{n}^{\lambda}|b_{n}|<\epsilon$. Since $\lambda_{0}$ can be chosen independently of $n$, the lemma is proved. 
\end{proof}

With the help of Lemma \ref{lemma2.2.3.50}, we can use the same method as in \cite[Theorem 4.3]{MR1964949} to prove the following lemma.

\begin{lemma}\label{lemmatightness}Let $\beta, \epsilon>0$. Then there exists $\delta>0$ not depending on $(s,x)$ and  $n$,  such that
\begin{equation}
    \mathbf{P}^{s,x}_{n}\Big(\sup_{s \le t \le s+\delta}|X_{t}-x|>\beta\Big)<\epsilon.
\end{equation}
\end{lemma}
\begin{proof}
As shown in the proof of \cite[Theorem 4.3]{MR1964949}, it suffices to find a $\delta>0$, which is independent of $(s,x)$ and  $n$, such that
\begin{equation}\label{eqsufficienttightness}
    \mathbf{P}^{s,x}_{n}\Big(\int_{s}^{s+\delta}|b_{n}(u,X_{u})|du > \frac{\beta}{2}\Big) < \frac{\epsilon}{2}.
\end{equation}
Let $\theta = \delta^{-1}$. Then
\begin{align*}\mathbf{P}^{s,x}_{n}\Big(\int_{s}^{s+\delta}|b_{n}(u,X_{u})|du  >  \frac{\beta}{2}\Big)
\le &  \frac{2}{\beta}\mathbf{E}^{s,x}_{n}\Big[\int_{s}^{s+\delta}|b_{n}(u,X_{u})|du\Big] \\
\le & \frac{2e}{\beta}\mathbf{E}^{s,x}_{n}\Big[\int_{s}^{s+\delta}e^{-\theta (u-s)}|b_{n}(u,X_{u})|du\Big] \\
\le & \frac{2e}{\beta}\mathbf{E}^{s,x}_{n}\Big[\int_{s}^{\infty}e^{-\theta (u-s)}|b_{n}(u,X_{u})|du \Big]\\
\le & \frac{2e}{\beta}S_{n}^{\theta}|b_{n}|(s,x).
\end{align*}
By Lemma \ref{lemma2.2.3.50}, we can find a large enough $\theta>0$, independent of $(s,x)$ and $n$, such that
\[
S_{n}^{\theta}|b_{n}|(s,x) < \frac{\beta \epsilon}{4e}.\]
Thus the assertion (\ref{eqsufficienttightness}) follows. 
\end{proof}

\begin{lemma}\label{lemma2.2.3.0}Let $\beta \in (0,1]$. Define $\sigma_{s}:= \inf \{r \ge 0:|X_{s+r}-X_{s}|\ge \beta\}.$ Then there exists $0<\delta<1$, which does not depend on $(s,x)$, such that
\[
    \sup_{n \in \BN}\mathbf{E}^{s,x}_{n}[e^{-\sigma_{s}}]\le \delta.
\]
\end{lemma}
\begin{proof}The proof goes in the same way as the proof of \cite[Corollary 4.4]{MR1964949}. 
\end{proof}

Next, we show how we can use the above two lemmas and the strong Markov property to deduce the tightness of $\{ \mathbf{P}^{s,x}_{n}:  n \in \BN \}$.
\begin{proposition}\label{newproptight}Let $\beta, \epsilon, T>0$. Then there exists $\delta>0$, which does not depend on $(s,x)$ and  $n$, such that
\begin{equation}\label{eqfortightness}
    \mathbf{P}^{s,x}_{n}\Big(\sup_{t, t'\le T, |t-t'|\le \delta}|X_{t}-X_{t'}|>\beta\Big)<\epsilon.
\end{equation}
\end{proposition}
\begin{proof}
Since $\mathbf{P}^{s,x}_{n}(X_t=x \mbox{ for all } t \le s)=1$, we can assume without loss of generality that $s<T$; otherwise
\[
\mathbf{P}^{s,x}_{n}\Big(\sup_{t, t'\le T, |t-t'|\le \delta}|X_{t}-X_{t'}|>\beta\Big)=0.
\]
Furthermore, it is easy to see that the inequality (\ref{eqfortightness}) is equivalent to
\[
\mathbf{P}^{s,x}_{n}\Big(\sup_{\substack{s\le t,  t'\le T \\ |t-t'|\le \delta}}|X_{t}-X_{t'}|>\beta\Big)<\epsilon.
\]

Let $\tau_{0}=s$ and define inductively
\[
\tau_{i+1}:=\inf \big\{t\ge \tau_{i}:|X_{t}-X_{\tau_i}|\ge \frac{\beta}{3} \big\}.
\]
By the strong Markov property,
\[
\mathbf{E}^{s,x}_{n}\big[e^{-\tau_{i+1}} \big ]=\mathbf{E}^{s,x}_{n} \big[e^{-(\tau_{i+1}-\tau_{i})}e^{-\tau_{i}} \big]=\mathbf{E}^{s,x}_{n} \big[e^{-\tau_{i}}\mathbf{E}_n^{\tau_{i},X_{\tau_{i}}}[e^{-(\tau_{i+1}-\tau_{i})}] \big].\]
From Lemma \ref{lemma2.2.3.0} we know that $\mathbf{E}_n^{\tau_{i},X_{\tau_{i}}}[e^{-(\tau_{i+1}-\tau_{i})}]<\gamma<1$, where $\gamma$ is a constant and is independent of $(\tau_{i},X_{\tau_{i}})$ and $n$. Therefore, by induction,
\[
\mathbf{E}^{s,x}_{n}\big[e^{-\tau_{i+1}} \big ] \le \gamma^{i+1}.\]
Thus
\begin{align}
\mathbf{P}^{s,x}_{n}(\tau_i \le T)=&\mathbf{P}^{s,x}_{n}(e^{-\tau_i} \ge e^{-T}) \notag \\
\label{estifortau}\le & e^T \mathbf{E}^{s,x}_{n}\big[e^{-\tau_{i}} \big ] \le  e^T \gamma^{i},
\end{align}
which implies the existence of $i_0 \in \BN$ with
\[
\mathbf{P}^{s,x}_{n}(\tau_{i_0} \le T)< \frac{\epsilon}{2}.
\]
It should be noted that $i_0$ is independent of $n$.

It remains to show that we can find a $\delta>0$, which does not depend on $(s,x)$ and  $n$, such that
\begin{equation}\label{neweqtightnesssm4}
    \mathbf{P}^{s,x}_{n}\Big( \tau_{i_0} > T \ \mbox{ and }\ \sup_{\substack{s\le t,  t'\le \tau_{i_0} \\ |t-t'|\le \delta}}|X_{t}-X_{t'}|>\beta\Big)\le \frac{\epsilon}{2}.
\end{equation}
Suppose now $\tau_{i_0}(\omega) > T$, ~$|X_{t}(\omega)-X_{t'}(\omega)|>\beta, ~ s\leq t < t'\leq \tau_{i_0}(\omega)$ and $|t-t'|\leq\delta$. Then there exists an $l\in\left\{0,1,\cdots,i_0-1\right\}$ such that
\[
 \tau_l(\omega) \leq t < \tau_{l+1}(\omega).\]
If $ \tau_l(\omega) \leq t < t'\leq \tau_{l+1}(\omega)$, then
\[
|X_{t}(\displaystyle \omega)-X_{\tau_l}(\omega)|\le \frac{\beta}{3} \quad \mbox{and} \quad  |X_{t'}(\displaystyle \omega)-X_{\tau_l}(\omega)|\le \frac{\beta}{3},
\]
so
$$
|X_{t}(\omega)-X_{t'}(\omega)|\leq\frac{2\beta}{3},
$$
which is a contradiction to the fact that $|X_{t}(\omega)-X_{t'}(\omega)|>\beta$.
Thus it must hold $ \tau_l(\omega) \leq t < \tau_{l+1}(\omega)< t' \le \tau_{i_0}$. So
\[
|X_{t}(\displaystyle \omega)-X_{\tau_l}(\omega)|\le \frac{\beta}{3}, \qquad |X_{\tau_{l+1}}(\displaystyle \omega)-X_{\tau_l}(\omega)|\le \frac{\beta}{3},
\]
and hence
\begin{align*}
|X_{t'}(\displaystyle \omega)-X_{\tau_{l+1}}(\omega)|
\ge & |X_{t'}(\displaystyle \omega)-X_{t}(\displaystyle \omega)|- |X_{t}(\displaystyle \omega)-X_{\tau_l}(\omega)|& \\
& \qquad  \qquad \qquad \qquad \qquad -|X_{\tau_{l+1}}(\displaystyle \omega)-X_{\tau_l}(\omega)| & \\
>& \frac{\beta}{3},&
\end{align*}
which implies $0\le \tau_{l+2}(\omega)-\tau_{l+1}(\omega) < t'-t \le \delta$. Therefore,
\begin{align}
& \Big \{ \tau_{i_0} > T \ \mbox{ and }\ \sup_{\substack{s\le t,  t'\le \tau_{i_0} \\ |t-t'|\le \delta}}|X_{t}-X_{t'}|>\beta \Big \} \notag \\
\label{neweqtightnesssm5} \subset & \big\{\tau_{i_0} > T \mbox{ and } \min_{1\le l \le i_0} (\tau_l-\tau_{l-1}) < \delta \big\}
 \subset  \big\{\min_{1\le l \le i_0} (\tau_l-\tau_{l-1} )< \delta \big\}.
\end{align}
By Lemma \ref{lemmatightness}, we can find a sufficiently small $\delta>0$ such that
\begin{equation}\label{neweqtightness}
 \mathbf{P}^{s,x}_{n}\Big(\sup_{s \le t \le s+\delta}|X_{t}-x|>\frac{\beta}{4}\Big)<\epsilon'
\end{equation}
for all $(s,x) \in \hs$ and $n \in \BN$, where $\epsilon':=1-(1-\epsilon/2)^{1/i_0}$.
By the strong Markov property,
\begin{align}
& \mathbf{P}^{s,x}_{n}(\min_{1\le l \le i_0} (\tau_l-\tau_{l-1} ) < \delta) \\
= & 1-\mathbf{P}^{s,x}_{n}(\min_{1\le l \le i_0} (\tau_l-\tau_{l-1} ) \ge \delta) \notag \\
= & 1-\mathbf{E}^{s,x}_{n}\big[\prod_{l=1}^{i_0}\mathbf{1}_{\{\tau_l-\tau_{l-1} \ge\delta\}} \big ] \notag\\
=& 1-\mathbf{E}^{s,x}_{n} \big[\prod_{l=1}^{i_0-1}\mathbf{1}_{\{\tau_l-\tau_{l-1} \ge\delta\}}\mathbf{E}_n^{\tau_{i_0-1},X_{\tau_{i_0-1}}}[\mathbf{1}_{\{\tau_{i_0}-\tau_{i_0-1}\ge\delta\}}] \big] \notag \\
\label{neweqtightnesssm}= & 1-\mathbf{E}^{s,x}_{n} \big[\prod_{l=1}^{i_0-1}\mathbf{1}_{\{\tau_l-\tau_{l-1}\ge\delta\}}\mathbf{P}_n^{\tau_{i_0-1},X_{\tau_{i_0-1}}}(\tau_{i_0}-\tau_{i_0-1}\ge\delta) \big].
\end{align}
It follows from (\ref{neweqtightness}) that
\begin{align}
& \mathbf{P}_n^{\tau_{i_0-1},X_{\tau_{i_0-1}}}(\tau_{i_0}-\tau_{i_0-1}\ge\delta) \notag \\ = &1-\mathbf{P}_n^{\tau_{i_0-1},X_{\tau_{i_0-1}}}(\tau_{i_0}-\tau_{i_0-1}<\delta) \notag\\
\ge & 1-\mathbf{P}_n^{\tau_{i_0-1},X_{\tau_{i_0-1}}}\Big(\sup_{\tau_{i_0-1} \le t \le \tau_{i_0-1}+\delta}|X_{t}-X_{\tau_{i_0-1}}|>\frac{\beta}{4}\Big) \notag\\
\label{neweqtightnesssm2}\ge & 1-\epsilon'.
\end{align}
Using (\ref{neweqtightnesssm}), (\ref{neweqtightnesssm2}) and induction, we obtain
\begin{align}
\mathbf{P}^{s,x}_{n}\Big(\min_{1\le l \le i_0} (\tau_l-\tau_{l-1} ) < \delta\Big) \le & 1-(1-\epsilon')\mathbf{E}^{s,x}_{n} \Big[\prod_{l=1}^{i_0-1}\mathbf{1}_{\{\tau_l-\tau_{l-1} \ge\delta\}} \Big]\notag \\
\label{neweqtightnesssm3}\le & 1-(1-\epsilon')^{i_0} = \frac{\epsilon}{2}.
\end{align}
Combining (\ref{neweqtightnesssm5}) and (\ref{neweqtightnesssm3}) yields (\ref{neweqtightnesssm4}).

\end{proof}

From now on and till the end of this section we fix $(s,x) \in \hs$. By Proposition \ref{newproptight} and \cite[Theorem 1.3.2]{MR2190038}, we know that the family $\{ \mathbf{P}^{s,x}_{n}: n\in \BN\}$ of probability measures on $\big(\Omega=C([0,\infty);\mathbb{R}^{d}),\mathcal{M}\big)$ is tight, so we can find a subsequence $( \mathbf{P}^{s,x}_{n_k})_{k \ge 1}$ which converges weakly. Suppose that $\mathbf{P}^{s,x}:=\lim_{k\to \infty}\mathbf{P}^{s,x}_{n_{k}}$ is the limit point. Then we have the following proposition, which establishes the connection between the probability measure $\mathbf{P}^{s,x}$ and the resolvent operator $S^{\lambda}$.

\begin{proposition}\label{thm2.3.200}For every bounded and measurable function $f$ on $\hs$, we have
\[
\mathbf{E}^{s,x}\Big[\int_{s}^{\infty}e^{-\lambda(t-s)}f(t,X_{t})dt\Big]=S^{\lambda}f(s,x),\qquad (s,x)\in \hs, \]
where $\mathbf{E}^{s,x}[\cdot]$ denotes the expectation with respect to the measure $\mathbf{P}^{s,x}$ and $S^{\lambda}$ is defined by (\ref{neweqdefiofslambda}). %Moreover, the above equation is also true for \textcolor{red}{non-negative and measurable} function $f$ on $\hs$.
\end{proposition}
\begin{proof}In view of Lemma \ref{newlemmasc} and the weak convergence of $ \mathbf{P}^{s,x}_{n_k}$ to $\mathbf{P}^{s,x}$ as $k \to \infty$, the assertion is clearly true when $f$ is bounded and continuous. For any open subset $G\subset \hs$, we can find $f_n \in C_b(\hs)$, $n \in \BN$, such that $0\le f_n \uparrow \mathbf{1}_{G}$ as $n \to \infty$. By dominated convergence theorem, the assertion also holds for $\mathbf{1}_{G}$. The general case then follows by a standard monotone class argument, see, for example, \cite[p.~4]{MR648601}. 
\end{proof}

We now show that $\mathbf{P}^{s,x}$ is a solution to the martingale problem for
\begin{equation}\label{defioflt}
    L_{t}=\frac{1}{2}\triangle + b(t,\cdot) \cdot \nabla.
\end{equation}

\begin{theorem}\label{solutiontothempp}The probability measure $\mathbf{P}^{s,x}$ is a solution to the martingale problem for $L_{t}$ starting from
$(s,x)$.
\end{theorem}
\begin{proof}
We need to show that
\[
    f(X_{t})-\int^{t}_{s}L_{u}f(X_{u})du
\]
is a $\mathbf{P}^{s,x}$-martingale (with respect to the filtration $(\CM_t)_{t\ge0})$) after time $s$ for each $ f \in C^{\infty}_{0}(\mathbb{R}^{d})$. Suppose $s \le t_{1} < t_{2}$, $0 \le r_{1} \le \cdots \le r_{l} \le t_{1}$ and $g_{1}, \cdots, g_{l} \in C_{0}(\mathbb{R}^{d})$ with $l \in \BN$. Let $Y:=\prod_{j=1}^{l}g_{j}(X_{r_{j}})$. It suffices to show that
\begin{equation}\label{eqexistencesuf}
   \mathbf{E}^{s,x} \Big[Y\big(f(X_{t_1})-\int^{t_1}_{s}L_{u}f(X_{u})du\big)\Big]=\mathbf{E}^{s,x} \Big[Y\big(f(X_{t_2})-\int^{t_2}_{s}L_{u}f(X_{u})du\big)\Big].
\end{equation}

Recall that $\mathbf{P}^{s,x}$ is the weak limit of a subsequence of $(\mathbf{P}^{s,x}_{n})_{n \in \BN}$. For simplicity, we denote this subsequence still by $(\mathbf{P}^{s,x}_{n})_{n \in \BN}$. Since $\mathbf{P}^{s,x}_{n}$ solves the martingale problem for \[
    \frac{1}{2}\triangle + b_{n}(t,\cdot) \cdot \nabla
\]
starting from $(s,x)$, so
\[
    f(X_{t})-\int^{t}_{s}\big(\frac{1}{2}\triangle f(X_{u}) +     b_{n}(u,X_{u}) \cdot \nabla f(X_{u})\big)du
\]
is a $\mathbf{P}^{s,x}_{n}$-martingale after time $s$. Therefore,
\begin{align}
   & \mathbf{E}^{s,x}_{n}\Big[Y\Big(f(X_{t_{1}})-\int^{t_{1}}_{s}\big(\frac{1}{2}\triangle f(X_{u}) +     b_{n}(u,X_{u}) \cdot \nabla f(X_{u})\big)du\Big)\Big] \notag \\
 \label{eqmartingale}   =&\mathbf{E}^{s,x}_{n}\Big[Y\Big(f(X_{t_{2}})-\int^{t_{2}}_{s}\big(\frac{1}{2}\triangle f(X_{u}) +     b_{n}(u,X_{u}) \cdot \nabla f(X_{u})\big)du\Big)\Big].
 \end{align}
By the weak convergence of $\mathbf{P}^{s,x}_{n}$ to $\mathbf{P}^{s,x}$ as $n \to \infty$, we have
\begin{equation}\label{eqthm1i12}
\lim_{n \to \infty}\mathbf{E}^{s,x}_{n}\Big[Y\big(f(X_{t_{i}})-\int^{t_{i}}_{s}\frac{1}{2} \triangle f(X_{u})du \big) \Big]=\mathbf{E}^{s,x}\Big[Y\big(f(X_{t_{i}})-\int^{t_{i}}_{s}\frac{1}{2} \triangle f(X_{u})du\big) \Big]
\end{equation}
for $i=1,2$. If we can show that
\begin{equation}\label{eqthm1main}
    \lim_{n \to \infty}\mathbf{E}^{s,x}_{n}\Big[Y\int^{t_{i}}_{s}b_{n}(u,X_{u})\cdot \nabla f(X_{u})du\Big]=\mathbf{E}^{s,x}\Big[Y\int^{t_{i}}_{s}b(u,X_{u})\cdot \nabla f(X_{u})du\Big]
    \end{equation}
for $i =1,2$, then (\ref{eqexistencesuf}) follows from (\ref{eqmartingale}), (\ref{eqthm1i12}) and (\ref{eqthm1main}). Next we show that (\ref{eqthm1main}) is true.

According to (\ref{rlambdabn}) and (\ref{estiforsnm}), we have for each $k \ge 1$
\begin{align}
& \mathbf{E}^{s,x}_{k}\Big[\int^{t_{1}}_{s}(|b_{n}-b_m|)(u,X_{u})du\Big]\notag\\
\le & e^{\lambda(t_1-s)}  S_{k}^{\lambda}(|b_{n}-b_{m}|)(s,x)\notag \\
\le & e^{\lambda(t_1-s)} R^{\lambda}(|b_{n}-b_{m}|)(s,x)+ \frac{1}{2}C_{1}e^{\lambda(t_1-s)}N^{\alpha,+}_{2\epsilon_{1}}(|b_{n}-b_{m}|)  R^{\lambda} (|b_{k}|)\notag\\
\le &  \frac{5}{4}e^{\lambda(t_1-s)}N^{\alpha,+}_{2\epsilon_{1}}(|b_{n}-b_{m}|)  \notag\\
\label{neweqek}\le &  \frac{5}{4}e^{\lambda(t_1-s)} N^{\alpha,+}_{2\epsilon_{1}}(|b_{n}-b|+|b_{m}-b|) \to 0, \quad  \text{as} \ n,m \to \infty.
\end{align}
Similarly to (\ref{neweqek}), we obtain
\begin{align}
& \mathbf{E}^{s,x}\Big[\int^{t_{1}}_{s}(|b_{n}-b|)(u,X_{u})du\Big] \notag\\
%& \bigg|\mathbf{E}^{s,x}\Big[Y\int^{t_{1}}_{s}b^{i}_{n}(u,x)\frac{\partial f }{\partial x_{i}}(X_{u})du\Big]-\mathbf{E}^{s,x}\Big[Y\int^{t_{1}}_{s}b^{i}(u,x)\frac{\partial f }{\partial x_{i}}(X_{u})du\Big] \bigg|\notag\\
\label{neweqesx}\le &  \frac{5}{4}e^{\lambda(t_1-s)}N^{\alpha,+}_{2\epsilon_{1}}(|b_{n}-b|) \to 0, \quad \text{as} \ n \to \infty.
\end{align}
By (\ref{neweqek}) and (\ref{neweqesx}), for any given $\epsilon>0$, we can find $n_{1} \in \BN$, which is independent of $k$, such that for all $n,m \ge n_{1}$
\[
\bigg|\mathbf{E}^{s,x}_{k}\Big[Y\int^{t_{1}}_{s}b_{n}(u,X_{u})\cdot \nabla f(X_{u})du\Big]-\mathbf{E}^{s,x}_{k}\Big[Y\int^{t_{1}}_{s}b_{m}(u,X_{u})\cdot \nabla f(X_{u})du\Big] \bigg|<\epsilon \]
and
\[
\bigg|\mathbf{E}^{s,x}\Big[Y\int^{t_{1}}_{s}b_{n}(u,X_{u})\cdot \nabla f(X_{u})du\Big]-\mathbf{E}^{s,x}\Big[Y\int^{t_{1}}_{s}b(u,X_{u})\cdot \nabla f(X_{u})du\Big] \bigg|<\epsilon.\]
Note that there exists $n_{2}$ such that for $n \ge n_{2}$
\[
\bigg|\mathbf{E}^{s,x}_{n}\Big[Y\int^{t_{1}}_{s}b_{n_{1}}(u,X_{u})\cdot \nabla f(X_{u})du\Big]-\mathbf{E}^{s,x}\Big[Y\int^{t_{1}}_{s}b_{n_{1}}(u,X_{u})\cdot \nabla f(X_{u})du\Big] \bigg|<\epsilon.\]
If $n \ge \max \{n_{1},n_{2}\}$, then
\begin{align*}&\bigg| \mathbf{E}^{s,x}_{n}\Big[Y\int^{t_{1}}_{s}b_{n}(u,X_{u})\cdot \nabla f(X_{u})du\Big]-\mathbf{E}^{s,x}\Big[Y\int^{t_{1}}_{s}b(u,X_{u})\cdot \nabla f(X_{u})du\Big] \bigg|\\
\le &\bigg|\mathbf{E}^{s,x}_{n}\Big[Y\int^{t_{1}}_{s}b_{n}(u,X_{u})\cdot \nabla f(X_{u})du\Big]-\mathbf{E}^{s,x}_{n}\Big[Y\int^{t_{1}}_{s}b_{n_{1}}(u,X_{u})\cdot \nabla f(X_{u})du\Big] \bigg|\\
&\ +\bigg|\mathbf{E}^{s,x}_{n}\Big[Y\int^{t_{1}}_{s}b_{n_{1}}(u,X_{u})\cdot \nabla f(X_{u})du\Big]-\mathbf{E}^{s,x}\Big[Y\int^{t_{1}}_{s}b_{n_{1}}(u,X_{u})\cdot \nabla f(X_{u})du\Big] \bigg|\\
& \quad + \bigg|\mathbf{E}^{s,x}\Big[Y\int^{t_{1}}_{s}b_{n_{1}}(u,X_{u})\cdot \nabla f(X_{u})du\Big]-\mathbf{E}^{s,x}\Big[Y\int^{t_{1}}_{s}b(u,X_{u})\cdot \nabla f(X_{u})du\Big] \bigg|\\
\le & 3 \epsilon.
\end{align*}
Therefore,
\[
    \lim_{n \to \infty}\mathbf{E}^{s,x}_{n}\Big[Y\int^{t_{1}}_{s}b_{n}(u,X_{u})\cdot \nabla f(X_{u})du\Big]=\mathbf{E}^{s,x}\Big[Y\int^{t_{1}}_{s}b(u,X_{u})\cdot \nabla f(X_{u})du\Big].\]
Similarly,
\[
    \lim_{n \to \infty}\mathbf{E}^{s,x}_{n}\Big[Y\int^{t_{2}}_{s}b_{n}(u,X_{u})\cdot \nabla f(X_{u})du\Big]=\mathbf{E}^{s,x}\Big[Y\int^{t_{2}}_{s}b(u,X_{u})\cdot \nabla f(X_{u})du\Big].\]
Thus (\ref{eqthm1main}) is true. This completes the proof. 
\end{proof}

If there is another subsequence $( \mathbf{P}^{s,x}_{\tilde{n}_k})_{k \ge 1}$ of $( \mathbf{P}^{s,x}_{n})_{n \ge 1}$ which converges weakly to a probability measure $\tilde{\mathbf{P}}^{s,x}$, then $\tilde{\mathbf{P}}^{s,x}$ is obviously also a solution to the martingale problem for $L_{t}$ starting from $(s,x)$. Now we proceed to show that $\tilde{\mathbf{P}}^{s,x}=\mathbf{P}^{s,x}$.

The following lemma is a variant of \cite[Theorem~6.1.3]{MR2190038} and plays an important role in showing the uniqueness of solutions to the martingale problem for $L_t$. It should be noted that the boundedness of the drift $b$ was needed in \cite[Theorem~6.1.3]{MR2190038}. For our case this restriction can be dropped. Since the proof is almost the same as that of \cite[Theorem~6.1.3]{MR2190038}, we put it in the appendix.

\begin{lemma}\label{newlemmamgt}Suppose that the probability measure $\mathbf{Q}^{s,x}$ on $(\Omega=C([0,\infty);\mathbb{R}^{d}),\mathcal{M})$ is a solution to the martingale problem for $L_{t}$ starting from $(s,x)$, where $L_t$ is defined by (\ref{defioflt}). For a given $t\ge s$, we denote by $Q_{\omega}(A)=Q(\omega,A): \Omega \times \CM \to [0,1]$ the regular conditional distribution of $\mathbf{Q}^{s,x}$ given $\mathcal{M}_{t}$. Then there exists a $\mathbf{Q}^{s,x}$-null set $N\in
\mathcal{M}_{t}$ such that $Q_{\omega}$ solves the martingale problem for $L_{t}$ starting from $(t, \omega_t)$ for each $\omega \notin N$.
\end{lemma}

\begin{proposition}\label{unilimitofsub}
If $\tilde{\mathbf{P}}^{s,x}$ is another limit point of  $\{ \mathbf{P}^{s,x}_{n}:n \in \BN\}$ under the topology of weak convergence for measures, then we have $\tilde{\mathbf{P}}^{s,x}=\mathbf{P}^{s,x}$.
\end{proposition}
\begin{proof}Let $\tilde{\mathbf{E}}^{s,x}[\cdot]$ denote the expectation with respect to the measure $\tilde{\mathbf{P}}^{s,x}$. According to Proposition \ref{thm2.3.200} and the uniqueness of the Laplace transform, we have
\[
\tilde{\mathbf{E}}^{s,x}[f(X_{t})]=\mathbf{E}^{s,x}[f(X_{t})], \quad \forall f\in C_{b}(\mathbb{R}^{d}), \ t \ge 0.
\]
It means that one-dimensional distributions of $\tilde{\mathbf{P}}^{s,x}$ and $\mathbf{P}^{s,x}$ coincide.  Since $\tilde{\mathbf{P}}^{s,x}$ and $\mathbf{P}^{s,x}$ are both solutions to the martingale problem for $L_{t}$, we can use Lemma \ref{newlemmamgt} and the standard argument in the proof of \cite[Theorem~6.2.3]{MR2190038} to show that multi-dimensional distributions of $\tilde{\mathbf{P}}^{s,x}$ and $\mathbf{P}^{s,x}$ are also the same.  Thus $\tilde{\mathbf{P}}^{s,x}=\mathbf{P}^{s,x}$ on $(\Omega,\mathcal{M})$. 
\end{proof}

\begin{cor}\label{psxismeasurable}
The sequence $( \mathbf{P}^{s,x}_{n})_{n \ge 1}$ converges weakly to $\mathbf{P}^{s,x}$. Consequently, the family of measures $\{ \mathbf{P}^{s,x}: \ (s,x) \in \hs\}$ is measurable, that is, $\mathbf{P}^{s,x}(A)$ is measurable in $(s,x) $ for every $A\in \CM$.
\end{cor}
\begin{proof}The weak convergence of $(\mathbf{P}^{s,x}_{n})_{n \ge 1}$ to $ \mathbf{P}^{s,x}$ follows from Proposition \ref{unilimitofsub} and the fact that $ \{\mathbf{P}^{s,x}_n: n \in \BN\}$ is tight. Let $f\in C_{b}(\mathbb{R}^{d})$ and $ t \ge 0$. Since $b_n$ is smooth and of compact support, the function $\mathbf{E}^{s,x}_n[f(X_{t})]$ is measurable in $(s,x)$. By the weak convergence of $( \mathbf{P}^{s,x}_{n})_{n \ge 1}$ to $ \mathbf{P}^{s,x}$, the function $\mathbf{E}^{s,x}[f(X_{t})]$ is the limit of $\mathbf{E}^{s,x}_n[f(X_{t})]$ and thus also measurable in $(s,x)$. Since $\CM=\sigma(f(X_{t}): f\in C_{b}(\mathbb{R}^{d}), \ t \ge 0)$, the assertion follows. 
\end{proof}

We now prove the uniqueness of solutions to the martingale problem for $L_{t}$. This can be done in the same way as in \cite{MR1964949}. The first step is to show that one-dimensional distributions of solutions to the martingale problem are unique.

\begin{proposition}\label{thm2.2.3.1} Let $\mathbf{P}^{s,x}$ be the solution to the martingale problem for $L_t$ that we derived in Theorem \ref{solutiontothempp}. If there exists another probability measure $\mathbf{Q}^{s,x}$ that solves the martingale problem for $L_{t}$ starting from $(s,x)$, then for all $f\in C_{b}(\mathbb{R}^{d})$ we have
\begin{equation}\label{BSTeqthm2.2.3.1.0}
\mathbf{E}_{\mathbf{Q}^{s,x}}[f(X_{t})]=\mathbf{E}_{\mathbf{P}^{s,x}}[f(X_{t})], \quad \forall t \ge 0.
\end{equation}
\end{proposition}
\begin{proof}Our proof is adapted from the proof of  \cite[Proposition 5.1]{MR1964949}. Since the details of the proof may obscure the idea, we now outline the reasoning behind our analysis. We will
first introduce a sequence of stopping times $(\tau_{n})_{n \in \BN}$ that converges a.s. to $\infty$ under the measure $\mathbf{Q}^{s,x}$. Next, we ``glue" the measures $\mathbf{Q}^{s,x}$ and
$\mathbf{P}^{ \tau_{n}, X_{\tau_{n}}}$ at the stopping time $\tau_n$. In this way, we obtain a new measure $\mathbf{Q}^{s,x}_{n}$. Roughly speaking, under the measure $\mathbf{Q}^{s,x}_{n}$, the canonical process $(X_t)_{t\ge 0}$ on the path space  behaves according to
$\mathbf{Q}^{s,x}$ before $\tau_n$ and then according to $\mathbf{P}^{ \tau_{n}, X_{\tau_{n}}}$ after $\tau_n$. By the introduction of $\tau_n$, we obtain the inequality (\ref{newneweqthm2finfty3})
(see below) for $\mathbf{Q}^{s,x}_{n}$, which allows us to conveniently use the theorems of Fubini and dominated convergence. By using the standard argument, we  then show that $\mathbf{Q}^{s,x}_{n}$ is also a solution to the martingale problem for
$L_t$ starting from $(s,x)$ and, further,
\[
\mathbf{E}_{\mathbf{Q}^{s,x}_{n}}[f(X_{t})]=\mathbf{E}_{\mathbf{P}^{s,x}}[f(X_{t})],\quad \forall f\in C_{b}(\mathbb{R}^{d}),\ t \ge 0.\]
With $n \to \infty$, we get (\ref{BSTeqthm2.2.3.1.0}).

We now proceed to prove this proposition. Let $(\mathcal{F}_{t})_{t \ge0}$ be the usual augmentation of $(\CM_{t})_{t \ge0}$ with respect to $\mathbf{Q}^{s,x}$. Define a sequence of $\CF_t$-stopping times
\[\sigma_{n}:=\inf\{t\ge s: \int_{s}^{t}|b(u,X_{u})|du > n\}, \quad n \in \BN,  \]
and let
\[
 \tau_{n}:=\sigma_{n}\land n, \quad n \in \BN \ \ \mbox{with} \ \ n \ge s.
\]
According to the condition (\ref{wsintegrcondi}), it is easy to see that $\tau_n \to \infty$ $\ \mathbf{Q}^{s,x}$-a.s.

For each fixed $\omega \in \Omega$, it follows from \cite[Lemma 6.1.1]{MR2190038} that there is a unique probability measure $\delta_{\omega} \bigotimes_{ \tau_{n}(\omega)}\mathbf{P}^{ \tau_{n}(\omega), \omega_{\tau_{n}}}$ on $(\Omega, \CM)$ such that
\[
\textstyle \delta_{\omega} \bigotimes_{ \tau_{n}(\omega)}\mathbf{P}^{ \tau_{n}(\omega), \omega_{\tau_{n}}}\big(X_t=\omega_t, \ 0\le t \le \tau_n(\omega)\big)=1 \]
and
\[
\textstyle \delta_{\omega} \bigotimes_{ \tau_{n}(\omega)}\mathbf{P}^{ \tau_{n}(\omega), \omega_{\tau_{n}}}(A)=\mathbf{P}^{ \tau_{n}(\omega),\omega_{\tau_{n}}}(A), \quad A \in \CM^{\tau_{n}(\omega)},
\]
where $\CM^{t}:=\sigma(X(r): r\ge t)$, $t \ge0$. In view of Corollary \ref{psxismeasurable}, it is easy to check that $\delta_{(\cdot)} \bigotimes_{ \tau_{n}(\cdot)}\mathbf{P}^{ \tau_{n}(\cdot), (\cdot)_{\tau_{n}}}$  is a probability kernel from $(\Omega, \CF_{\tau_n})$ to $(\Omega,\CM)$. Thus it induces a probability measure $\mathbf{Q}^{s,x}_{n}$ on $(\Omega, \CM)$ with
\[
    \mathbf{Q}^{s,x}_{n}(A)= \int_{\Omega}\textstyle\delta_{\omega} \bigotimes_{ \tau_{n}(\omega)}\mathbf{P}^{ \tau_{n}(\omega), \omega_{\tau_{n}}}(A)\mathbf{Q}^{s,x}(d\omega), \quad   A \in \mathcal{M}.\]
As done in the proof of \cite[Theorem 6.1.2]{MR2190038}, it is easy to check that $\mathbf{Q}^{s,x}_{n}$ is again a solution to the martingale problem for $L_{t}$ starting from $(s,x)$.
Moreover,
\begin{align}&  \mathbf{E}_{\mathbf{Q}^{s,x}_{n}}\Big[\int_{s}^{\infty} e^{-\lambda (t-s)}|b(t,X_{t})|dt\Big] \notag \\
 =& \mathbf{E}_{\mathbf{Q}^{s,x}}\Big[\int_{s}^{\tau_{n}} e^{-\lambda (t-s)}|b(t,X_{t})|dt\Big]\notag \\
&\quad +\mathbf{E}_{\mathbf{Q}^{s,x}}\Big[e^{-\lambda (\tau_{n}-s)}\mathbf{E}_{\mathbf{P}^{\tau_{n},X_{\tau_{n}}}}\big[\int_{\tau_{n}}^{\infty} e^{-\lambda (t-\tau_{n})}|b(t,X_{t})|dt\big]\Big]\notag\\
\label{eqthm2infty1}\le & n+\mathbf{E}_{\mathbf{Q}^{s,x}}\Big[S^{\lambda}|b|(\tau_{n},X_{\tau_{n}})\Big].
 \end{align}
Similarly to (\ref{estiforsnm2}), we have
\begin{equation}\label{eqthm2infty2}
S^{\lambda}|b| \le 2R^{\lambda}|b|\le 2 N^{\alpha,+}_{2\epsilon_{1}}(|b|)< \infty.
\end{equation}
It follows from (\ref{eqthm2infty1}) and (\ref{eqthm2infty2}) that
\begin{equation}\label{newneweqthm2finfty3}
\mathbf{E}_{\mathbf{Q}^{s,x}_{n}}\Big[\int_{s}^{\infty} e^{-\lambda (t-s)}|b(t,X_{t})|dt\Big] <\infty.
\end{equation}

Under the probability measure $\mathbf{Q}^{s,x}_{n}$, the process $(X_t -\int_{s}^{t} b(u,X_{u})du )_{t \ge s}$ is a Brownian motion after time $s$. Applying  Ito's formula for $f\in C^{1,2}_b(\hs)$, we obtain
\begin{align*}& f(t,X_{t})-f(s,X_{s}) \\
=& \int_{s}^{t}\nabla f(u,X_{u})\cdot dX_{u}+\int_{s}^{t}\frac{\partial f}{\partial u}(u,X_{u})du+\frac{1}{2}\int_{s}^{t}\triangle f(u,X_{u})du\\
=& ``Martingale"+\int_{s}^{t}(\frac{\partial f}{\partial u}+L_{u}f)(u,X_{u})du.
\end{align*}
Taking expectations of both sides with respect to the measure $\mathbf{Q}^{s,x}_{n}$ gives
\[
    \mathbf{E}_{\mathbf{Q}^{s,x}_{n}}[f(t,X_{t})]-f(s,x)= \mathbf{E}_{\mathbf{Q}^{s,x}_{n}}\Big[\int_{s}^{t}(\frac{\partial f}{\partial u}+L_{u}f)(u,X_{u})du\Big]. \]
Multiplying both sides by $e^{-\lambda(t-s)}$, integrating with respect to $t$ from $s$ to $\infty$ and then applying Fubini's theorem, we get
\begin{align*}&  \mathbf{E}_{\mathbf{Q}^{s,x}_{n}}\Big[\int_{s}^{\infty} e^{-\lambda (t-s)}f(t,X_{t})dt\Big]\\
=&\frac{1}{\lambda}f(s,x)+  \mathbf{E}_{\mathbf{Q}^{s,x}_{n}}\Big[\int_{s}^{\infty} e^{-\lambda (t-s)}\int_{s}^{t}\big(\frac{\partial f}{\partial u}+L_{u}f\big)(u,X_{u})dudt\Big]\\
=& \frac{1}{\lambda}f(s,x)+\frac{1}{\lambda} \mathbf{E}_{\mathbf{Q}^{s,x}_{n}}\Big[\int_{s}^{\infty} e^{-\lambda (t-s)}\big(\frac{\partial f}{\partial t}+L_{t}f\big)(t,X_{t})dt\Big].
\end{align*}
Define a linear functional $V_{n}^{\lambda}$ by
\[
V_{n}^{\lambda}f:=\mathbf{E}_{\mathbf{Q}^{s,x}_{n}}\Big[\int_{s}^{\infty} e^{-\lambda (t-s)}f(t,X_{t})dt\Big] \]
for measurable functions $f$ on $\hs$ with
\[
\mathbf{E}_{\mathbf{Q}^{s,x}_{n}}\Big[\int_{s}^{\infty} e^{-\lambda (t-s)}f(t,X_{t})dt\Big]<\infty.
\]
Then
\begin{equation}\label{thmunieq1}
 \lambda V_{n}^{\lambda}f=f(s,x)+V_{n}^{\lambda}\Big(\frac{\partial f}{\partial t}+L_{t}f\Big), \quad f \in C^{1,2}_b(\hs). \end{equation}
For a given $g \in C^{1,2}_{b}(\hs)$, we have $f:=R^{\lambda}g \in C^{1,2}_b(\hs)$; moreover, since $R^{\lambda}$ is the space-time resolvent of Brownian motion, it holds
\begin{equation}\label{thmunieqnewnew}
   \lambda f(t,y) -\frac{1}{2}\triangle f(t,y)-\frac{\partial }{\partial t}f(t,y)=g(t,y), \quad (t,y) \in \hs.
\end{equation}
Substituting (\ref{thmunieqnewnew}) in the equation (\ref{thmunieq1}) and noting $f=R^{\lambda}g$,
we obtain
\begin{align*}
\lambda V_{n}^{\lambda}(R^{\lambda}g)=& R^{\lambda}g(s,x)+V_{n}^{\lambda}(\lambda f-g+BR^{\lambda}g ) \\
=& R^{\lambda}g(s,x)+V_{n}^{\lambda}(\lambda R^{\lambda}g-g+BR^{\lambda}g ),
\end{align*}
where $BR^{\lambda}$ is defined by (\ref{defiofBR}). Therefore,
\begin{equation}\label{neweqvn}
    V_{n}^{\lambda}g=R^{\lambda}g(s,x)+V_{n}^{\lambda}BR^{\lambda}g, \quad g \in C^{1,2}_b(\hs).
\end{equation}
 After a standard approximation procedure, the equation (\ref{neweqvn}) holds for any $g \in C_b(\hs)$. For any open subset $G\subset \hs$, we can find $g_k \in C_b(\hs)$, $k \in \BN$, such that
 $0\le g_k \uparrow \mathbf{1}_{G}$ as $k \to \infty$. For each $k \in \BN$, it holds
 \begin{equation}\label{BMSTDDPJeqnew}
 V_{n}^{\lambda}g_k=R^{\lambda}g_k(s,x)+V_{n}^{\lambda}BR^{\lambda}g_k.
 \end{equation}
 It's easy to see that $BR^{\lambda}g_k \to BR^{\lambda}(\mathbf{1}_{G})$ pointwise as $k \to \infty$ and
 \[
 |BR^{\lambda}g_k(t,X_{t})|\le C_{\lambda}\|g_k\|_{\infty}|b(t,X_{t})|\le C_{\lambda}|b(t,X_{t})|.
 \]
 Noting (\ref{newneweqthm2finfty3}) and letting $k\to \infty$ in (\ref{BMSTDDPJeqnew}), we conclude from dominated convergence theorem  that (\ref{neweqvn}) also holds for $g=\mathbf{1}_{G}$.
 Now, a standard monotone class argument extends (\ref{neweqvn}) to every  $g \in \CB_b(\hs)$, see, for example, \cite[p.~4]{MR648601}. So,
\begin{equation}\label{newnewneweqvn}
    V_{n}^{\lambda}g=R^{\lambda}g(s,x)+V_{n}^{\lambda}BR^{\lambda}g, \quad g \in \CB_b(\hs).
\end{equation}

Let $g \in \CB_b(\hs)$. Set $A_k:=\{(t,y)\in \hs: |b(t,y)|\le k\}$, $k\in \BN$.  Since $|\nabla R^{\lambda}g|$ is bounded and $V_{n}^{\lambda}(|b|)=\mathbf{E}_{\mathbf{Q}^{s,x}_{n}}[\int_{s}^{\infty} e^{-\lambda (t-s)}|b(t,X_{t})|dt]< \infty$, by dominated convergence theorem, we obtain
\begin{align}
V_{n}^{\lambda}BR^{\lambda}g= &\mathbf{E}_{\mathbf{Q}^{s,x}_{n}}\Big[\int_s^{\infty}e^{-\lambda (t-s)}(b\cdot \nabla R^{\lambda}g)(t,X_t) dt\Big] \notag \\
=&\lim_{k \to \infty}\mathbf{E}_{\mathbf{Q}^{s,x}_{n}}\Big[\int_s^{\infty}e^{-\lambda (t-s)}(\mathbf{1}_{A_k}b\cdot \nabla R^{\lambda}g)(t,X_t) dt\Big] \notag \\
\label{newnewneweq}=& \lim_{k \to \infty} V_{n}^{\lambda}(\mathbf{1}_{A_k}b\cdot \nabla R^{\lambda}g).
\end{align}
By (\ref{eqthm2infty2}), we have $\|R^{\lambda}(|b|)\|_{\infty}<\infty$. Similarly to (\ref{newnewneweq}), we get
\begin{equation}\label{newnewnew11}
R^{\lambda}BR^{\lambda}g(s,x)=\lim_{k \to \infty}R^{\lambda}(\mathbf{1}_{A_k}b\cdot \nabla R^{\lambda}g)(s,x)
\end{equation}
and
\begin{equation}\label{newnewnew22}
V_{n}^{\lambda}(BR^{\lambda})^2g=\lim_{k \to \infty}V_{n}^{\lambda}BR^{\lambda}(\mathbf{1}_{A_k}b\cdot \nabla R^{\lambda}g).
\end{equation}
Since $|\mathbf{1}_{A_k}b\cdot \nabla R^{\lambda}g|$ is bounded, it follows from (\ref{newnewneweqvn}) that
\begin{equation}\label{newnewnew33}
    V_{n}^{\lambda}(\mathbf{1}_{A_k}b\cdot \nabla R^{\lambda}g)=R^{\lambda}(\mathbf{1}_{A_k}b\cdot \nabla R^{\lambda}g)(s,x)+V_{n}^{\lambda}BR^{\lambda}(\mathbf{1}_{A_k}b\cdot \nabla R^{\lambda}g).
\end{equation}
Letting $k\to \infty$ in (\ref{newnewnew33}) and using (\ref{newnewneweq}), (\ref{newnewnew11}) and (\ref{newnewnew22}), we obtain
\[
    V_{n}^{\lambda}BR^{\lambda}g=R^{\lambda}BR^{\lambda}g(s,x)+V_{n}^{\lambda}(BR^{\lambda})^{2}g  ,\quad g \in \CB_b(\hs).\]
This and (\ref{newnewneweqvn}) imply
\[
    V_{n}^{\lambda}g=R^{\lambda}g(s,x)+R^{\lambda}BR^{\lambda}g(s,x)+V_{n}^{\lambda}(BR^{\lambda})^{2}g  ,\quad g \in \CB_b(\hs).\]
Proceeding as above, we obtain, for each $k\in \BN$,
\[
    V_{n}^{\lambda}g=\sum_{i=0}^{k}R^{\lambda}(BR^{\lambda})^{i}g(s,x)+ V_{n}^{\lambda}(BR^{\lambda})^{k+1}g ,\quad g \in \CB_b(\hs). \]
But
\[
            |V_{n}^{\lambda}(BR^{\lambda})^{k+1}g|\le  \|\nabla R^{\lambda}(BR^{\lambda})^{k}g\|_{\infty} V_{n}^{\lambda}|b| \to 0, \quad \mbox{as} \ \ k \to \infty,
            \]
where the convergence of $\|\nabla R^{\lambda}(BR^{\lambda})^{k}g\|_{\infty}$ to $0$ follows from (\ref{neweqnormofBRk}), Lemma \ref{lemma2.2.2.1} and Assumption \ref{assumption2.2.3}. Therefore,
\begin{align*}
     \mathbf{E}_{\mathbf{Q}^{s,x}_{n}}\Big[\int_{s}^{\infty} e^{-\lambda (t-s)}g(t,X_{t})dt\Big]=&V_{n}^{\lambda}g \\
     =&\sum_{i=0}^{\infty}R^{\lambda}(BR^{\lambda})^{i}g(s,x) \\
     =& \mathbf{E}_{\mathbf{P}^{s,x}}\Big[\int_{s}^{\infty} e^{-\lambda (t-s)}g(t,X_{t})dt\Big].\end{align*}
By the uniqueness of the Laplace transform, we have
\[
\mathbf{E}_{\mathbf{Q}^{s,x}_{n}}[f(X_{t})]=\mathbf{E}_{\mathbf{P}^{s,x}}[f(X_{t})],\quad \forall f\in C_{b}(\mathbb{R}^{d}),\ t \ge s.\]
Consequently,
\begin{align*}\mathbf{E}_{\mathbf{Q}^{s,x}}[f(X_{t})]=&\lim_{n \to \infty}\mathbf{E}_{\mathbf{Q}^{s,x}}[f(X_{t}); t<\tau_{n}]  \\
=& \lim_{n \to \infty} \mathbf{E}_{\mathbf{Q}^{s,x}_{n}}[f(X_{t}); t<\tau_{n}] \\
=& \lim_{n \to \infty} \mathbf{E}_{\mathbf{Q}^{s,x}_{n}}[f(X_{t}); t<\tau_{n}]+\lim_{n \to \infty} \mathbf{E}_{\mathbf{Q}^{s,x}_{n}}[f(X_{t}); t\ge \tau_{n}] \\
=& \mathbf{E}_{\mathbf{P}^{s,x}}[f(X_{t})].
\end{align*}
This completes the proof. 
\end{proof}

Since Proposition \ref{thm2.2.3.1} and Lemma \ref{newlemmamgt} hold, the uniqueness of solutions to the martingale problem for $L_{t}$ now follows by a standard argument.
\begin{theorem}The probability measure $\mathbf{P}^{s,x}$ on $(\Omega=C([0,\infty);\mathbb{R}^{d}),\mathcal{M})$ is the unique solution to the martingale problem for $L_{t}$ starting from $(s,x)$. Therefore, the martingale problem for
\[
  L_{t}=\frac{1}{2} \triangle+b(t,\cdot) \cdot \nabla \]is well-posed.
\end{theorem}
\begin{proof}The proof is the same as the proof of \cite[Theorem~6.2.3]{MR2190038}. 
\end{proof}

\section{Existence and Uniqueness of Weak Solutions: Global Case}
Under Assumption \ref{assumption2.2.3}, we have proved that the martingale problem for $L_{t}$ is well-posed. Now we consider the general case, namely we only assume that
\[
    |b| \in \CF \CK_{d-1}^{\alpha} \ \ \text{for some} \  \alpha \in (0,1/2).\]
The procedure to construct a global solution from local solutions is quite standard and is usually referred to as the ``glueing argument".

\begin{theorem}\label{thmmain2}If $|b| \in \CF \CK_{d-1}^{\alpha}$ for some $\alpha \in (0,1/2)$, then the martingale problem for
\[
    L_{t}=\frac{1}{2}\triangle+b(t,\cdot) \cdot \nabla \] is well-posed. Equivalently, the SDE (\ref{SDE1}) has a unique weak solution for each $(s,x) \in \hs$.
\end{theorem}
\begin{proof}Since $|b| \in \CF \CK_{d-1}^{\alpha}$, we can find a sufficiently small $\epsilon_{1}>0$ such that
\[
N^{\alpha,+}_{2\epsilon_{1}}(|b|)<\frac{1}{2\kappa C_{1}}.\]
For each $(s,x) \in \hs$, let \[R_{(s,x)}:=[s,s+\epsilon_{1}] \times \big \{y \in \Rd: |y-x|\le 1 \big \}.
\]
Define
\[
\tilde{b}_{(s,x)}(t,y):=\mathbf{1}_{R_{(s,x)}}(t,y)b(t,y), \quad (t,y) \in \hs.
\]
According to Theorem \ref{solutiontothempp}, there is a solution $\mathbf{P}^{s,x}$ to the martingale problem for the operator
\[
\frac{1}{2}\triangle+\tilde{b}_{(s,x)}(t,\cdot) \cdot \nabla\] starting from $(s,x)$.

Now fix $(s,x) \in \hs$ and let $T_{0}=s$. Define \[
T_{i+1}:=\inf\{t\ge T_{i}:(t,X_{t}) \notin R_{(T_{i},X_{T_{i}})} \}.\]  Let $\mathbf{Q}_{1}=\mathbf{P}^{s,x}$. %According to \cite[Theorem 6.1.2]{MR2190038}, we inductively
Due to \cite[Theorem 6.1.2]{MR2190038}, we can inductively define
\[
    \mathbf{Q}_{i+1}\big(A\cap (C \circ \theta_{T_{i}})\big):=\mathbf{E}_{\mathbf{Q}_{i}}\big[ \mathbf{P}^{T_{i},X_{T_{i}}}(C);A \big ],\quad  A \in \CM_{T_{i}},\  C \in \CM,\]
where $\theta_t$ are the usual shift operators on $\Omega=C([0,\infty);\mathbb{R}^{d})$ so that $\big (\theta_t(\omega)\big)_s=\omega_{s+t}$. It is clear that $\mathbf{Q}_{m}|_{\CM_{T_{k}}}=\mathbf{Q}_{k}|_{\CM_{T_{k}}}$ when $m \ge k$. Let
\[
\mathbf{Q}(A):=\mathbf{Q}_{k}(A), \quad  A \in \CM_{T_{k}},\ k \in \BN.\]
Similarly to (\ref{estifortau}), there exists $\delta \in (0,1)$ such that
\[
\mathbf{E}_{\mathbf{Q}_{i+1}}\big[e^{-T_{i+1}} \big ] \le \delta^{i+1}.\]
For each fixed $M>0$, we have
\[
\lim_{i \to \infty}\mathbf{Q}_{i}(T_{i}\le M)=\lim_{i \to \infty}\mathbf{Q}_{i}\big(e^{-T_{i}}\ge e^{-M}\big)\le \lim_{i \to \infty} e^{M}\mathbf{E}_{\mathbf{Q}_{i}}\big[e^{-T_{i}} \big ]=0.  \]
Thus it follows from \cite[Theorem~1.3.5]{MR2190038} that $\mathbf{Q}$ extends uniquely to a probability measure on $(\Omega,\CM)$. It is then routine to check that $\mathbf{Q}$ is a solution to the martingale problem for $L_t$ starting from $(s,x)$. The proof of the uniqueness part can be achieved by standard arguments, see \cite[Section~6.3]{MR1483890} or \cite[Section~6.6]{MR2190038}. 
\end{proof}

\section*{Appendix: Additional Proofs}

\hspace*{\parindent}$Proof \ of \ \emph{(\ref{appendixclaim1})}$. For $0<h<1$, we have
\begin{align*}
&  \int_{0}^{h}\int_{\mathbb{R}^{d}}\frac{1}{t^{\frac{d+1}{2}}}\exp\Big(-c \frac{|y|^{2}}{t}\Big)|\tilde{f}(t,y)|dydt \\
 = &  \int_{1-h}^{1}\int_{\mathbb{R}^{d}}\frac{1}{(1-s)^{\frac{d+1}{2}}}\exp(-c \frac{|x|^{2}}{1-s})|f(s,x)|dxds \\
=&\int_{1-h}^{1}\int_{\{x:~|x|\le 3d(1-s)^{\frac{1}{2}}\}}\frac{1}{(1-s)^{\frac{d+1}{2}}}\exp(-c \frac{|x|^{2}}{1-s})|f(s,x)|dxds \\
=&   \int_{1-h}^{1}\frac{1}{(1-s)\ln (1-s)^{-1}}ds \int_{\{x:~|x|\le 3d(1-s)^{\frac{1}{2}}\}}\frac{1}{(1-s)^{\frac{d}{2}}}\exp(-c \frac{|x|^{2}}{1-s})dx\\
=&   \int_{1-h}^{1}\frac{1}{(1-s)\ln (1-s)^{-1}}ds \int_{\{x':~|x^{\prime}|\le 3d\}}\exp(-c |x^{\prime}|^{2})dx,
\end{align*}
where we have used a change of variable $ x^{\prime}:=x(1-s)^{-\frac{1}{2}}$ in the last equality.

Since
\[
\int_{1-h}^{1}\frac{1}{(1-s)\ln (1-s)^{-1}}ds=\infty \quad \mbox{and} \quad \int_{\{x':~|x^{\prime}|\le 3d\}}\exp(-c |x^{\prime}|^{2})dx>0, \]
we have
\[
N^{c,+}_{h}(\tilde{f})\ge\int_{0}^{h}\int_{\mathbb{R}^{d}}\frac{1}{t^{\frac{d+1}{2}}}\exp\Big(-c \frac{|y|^{2}}{t}\Big)|\tilde{f}(t,y)|dydt=\infty,\]
which implies $\tilde{f} \notin \CF \CK_{d-1}^{c}$. Next we show that $ f \in \CF \CK^{c}_{d-1}$.

For $1/2 \le s<1$, we have
 \[\lim_{s \to 1}\int_{s}^{1}\int_{\mathbb{R}^{d}}\frac{1}{(t-s)^{\frac{d+1}{2}}}\exp(-c \frac{|y|^{2}}{t-s})|f(t,y)|dydt =0.\]
In fact, if we set $r:=(t-s)/(1-s)$, then
\begin{align*}
&\lim_{s \to 1}\int_{s}^{1}\int_{\mathbb{R}^{d}}\frac{1}{(t-s)^{\frac{d+1}{2}}}\exp(-c \frac{|y|^{2}}{t-s})|f(t,y)|dydt  \\
\le &\lim_{s \to 1}\int_{s}^{1}\frac{1}{(t-s)^{\frac{1}{2}}(1-t)^{\frac{1}{2}} \ln (1-t)^{-1}}dt\int_{\mathbb{R}^{d}}\frac{1}{(t-s)^{\frac{d}{2}}}\exp(-c \frac{|y|^{2}}{t-s})dy\\
= &\lim_{s \to 1}C \int_{s}^{1}\frac{1}{(t-s)^{\frac{1}{2}}(1-t)^{\frac{1}{2}} \ln (1-t)^{-1}}dt \\
= &\lim_{s \to 1}C \int_{0}^{1}\frac{1}{r^{\frac{1}{2}}(1-r)^{\frac{1}{2}} \ln[ (1-s)^{-1}(1-r)^{-1}]}dr\\
\le &\lim_{s \to 1}\frac{C}{\ln (1-s)^{-1}}\int_{0}^{1}\frac{1}{r^{\frac{1}{2}}(1-r)^{\frac{1}{2}}}dr=0,
\end{align*}
where $C:=\int_{\mathbb{R}^{d}}\exp(-c |y|^{2})dy$ is a constant.

Let $\epsilon>0$ be given. Then we can find an $s_{0} \in (1/2,1)$ such that
\[\int_{s}^{1}\int_{\mathbb{R}^{d}}\frac{1}{(t-s)^{\frac{d+1}{2}}}\exp(-c\frac{|y|^{2}}{t-s})|f(t,y)|dydt <\frac{\epsilon}{4},\quad s \in [ s_{0}, 1). \]
Set \[
C'=\sup_{\frac{1}{2}\le t \le \frac{s_{0}+1}{2}}\frac{1}{(1-t)^{\frac{1}{2}} \ln(1-t)^{-1}}.\]
Let $h_{0}>0$ be sufficiently small such that
\[
    \int_{0}^{h_{0}}\int _{\mathbb{R}^{d}}\frac{1}{t^{\frac{d+1}{2}}}\exp(-c\frac{|y|^{2}}{t})dydt <\frac{\epsilon}{4C'}.\]
Now suppose $2h<h_{0} \wedge (1-s_{0})$. For $0 \le s<s_{0}$, we have
\begin{eqnarray*}
    &&\int_{s}^{s+h}\int_{\mathbb{R}^{d}}\frac{1}{(t-s)^{\frac{d+1}{2}}}\exp(-c\frac{|y|^{2}}{t-s})|f(t,y)|dydt \\
    && \le\int_{s}^{s+h}\int_{\mathbb{R}^{d}}\frac{1}{(t-s)^{\frac{d+1}{2}}}\exp(-c \frac{|y|^{2}}{t-s}) C'dydt \\
 && \le C' \frac{\epsilon}{4C'}\le \frac{\epsilon}{4}.
\end{eqnarray*}
If $s\ge s_{0}$, then
\begin{eqnarray*}
    &&\int_{s}^{s+h}\int_{\mathbb{R}^{d}}\frac{1}{(t-s)^{\frac{d+1}{2}}}\exp(-c \frac{|y|^{2}}{t-s})|f(t,y)|dydt \\
    && \le \int_{s}^{1}\int_{\mathbb{R}^{d}}\frac{1}{(t-s)^{\frac{d+1}{2}}}\exp(-c \frac{|y|^{2}}{t-s})|f(t,y)|dydt < \frac{\epsilon}{4}.
\end{eqnarray*}
Therefore,
\begin{eqnarray*}
&&\sup_{(s,x)\in\hs}\int_{s}^{s+h}\int_{\mathbb{R}^{d}}\frac{1}{(t-s)^{\frac{d+1}{2}}}\exp(-c \frac{|x-y|^{2}}{t-s})|f(t,y)|dydt \\
&&\le\sup_{s\in [0,\infty)}\int_{s}^{s+h}\int_{\mathbb{R}^{d}}\frac{1}{(t-s)^{\frac{d+1}{2}}}\exp(-c\frac{|y|^{2}}{t-s})|f(t,y)|dydt< \epsilon.
\end{eqnarray*}Thus we have proved that $\lim_{h\to 0}N^{c,+}_{h}(f)=0$,  namely $f \in \CF \CK^{c}_{d-1}$. \qed

\vspace{\bigskipamount}
$Proof \ of \ Lemma \ \it{\ref{newlemmamgt}}$. We follow the proof of \cite[Theorem~6.1.3]{MR2190038}. Let $\big\{f_{n}:f_{n}\in C_{0}^{\infty}(\mathbb{R}^{d}), \ n \in \BN \big\}$ be dense in $C_{0}^{\infty}(\mathbb{R}^{d})$. By \cite[Theorem 1.2.10]{MR2190038}, for each $f_{n}$, there exists $N_{n}\in \mathcal{M}_{t}$ such that $\mathbf{Q}^{s,x}(N_{n})=0$ and
\[
    M_{f_{n}}(u):=f_{n}(X_{u})-f_{n}(X_{t})-\int^{u}_{t}L_rf_{n}(r,X_{r})dr\]
is a martingale after time $t$ with respect to $(\Omega,\mathcal{M}_{u},Q_{\omega})$ for each $\omega \notin N_{n}.$

Define
\[\sigma_{l}:=\inf\{u\ge s: \int_{s}^{u}|b(r,X_{r})|dr > l\}, \quad l \in \BN .\]
Since $\mathbf{Q}^{s,x}\big(\int_{s}^{u}|b(r,X_{r})|dr<\infty \mbox{ for all } u \ge s\big)=1$, we have \[\mathbf{Q}^{s,x}(\sigma _{l}\to \infty \mbox{ as } l \to \infty)=1.\] Therefore, there exists $N_{\sigma} \in \CM_{t}$ such that
\[Q_{\omega}(\sigma _{l}\to \infty \mbox{ as } l \to \infty)=1 \quad\text{for all} \  \omega \notin N_{\sigma}.\]
Let $N:=N_{\sigma}\cup (\cup_{n \ge 1} N_{n}).$

We now fix $\omega \in \Omega \setminus N$. Let $(\CF_u)_{u\ge0}$ be the usual augmentation under $Q_{\omega}$ of the canonical filtration $(\CM_u)_{u \ge 0}$. Since $M_{f_{n}}$ has  $Q_{\omega}$-a.s. continuous paths, it is easy to verify that the process $M_{f_{n}}(u)$ is a martingale after time $t$ with respect to $(\Omega, \mathcal{F}_{u}, Q_{\omega})$. Since $\sigma_l$ is an $\CF_u$-stopping time, it follows that $M_{f_{n}}(u\land \sigma_{l})$ is also a martingale after time $t$ with respect to $(\Omega, \mathcal{F}_{u}, Q_{\omega})$. For any $f \in C^{\infty}_{0}(\mathbb{R}^{d})$, we can find $f_{n_{k}}$ such that $f_{n_{k}} \to f$ in $C^{\infty}_{0}(\mathbb{R}^{d})$ as $k \to \infty$. Then
\[
    M_{f_{n_{k}}}(u \land \sigma_{l}) \to M_{f}(u \land \sigma_{l}) \]
boundedly and $Q_{\omega}$-a.s. as $k \to \infty$. By dominated convergence theorem and the martingale property of $M_{f_{n_{k}}}(u \land \sigma_{l})$, we know that $(M_{f}(u\land \sigma_{l}),\mathcal{F}_{u},Q_{\omega})$ is also a martingale after $t$. Since $Q_{\omega}(\sigma _{l}\to \infty \mbox{ as } l \to \infty)=1$, the measure $Q_{\omega}$ solves the local martingale problem for $L_{t}$ starting from $(t,\omega_{t})$. Noting that the second order term in $L_{t}$ is $\triangle/2$, it follows from \cite[Proposition~5.4.11]{MR1121940} that the local martingale problem for $L_{t}$ is equivalent to the martingale problem for $L_{t}$. Thus $Q_{\omega}$ solves the martingale problem for $L_{t}$ starting from $(t,\omega_{t})$. \qed

%%%%%% Section 7: Acknowledgements
\par\bigskip\noindent
{\bf Acknowledgements.} The author would like to thank Professor Michael R\"ockner for many helpful discussions and useful suggestions on the first version
of this paper. The  valuable comments and suggestions from the referees are also gratefully acknowledged.

\bibliographystyle{amsplain}

\end{document}